\pdfoutput=1
\RequirePackage{ifpdf}
\ifpdf 
\documentclass[pdftex]{sigma}
\else
\documentclass{sigma}
\fi

\numberwithin{equation}{section}

\newtheorem{Theorem}{Theorem}[section]
\newtheorem*{Theorem*}{Theorem}
\newtheorem{Corollary}[Theorem]{Corollary}
\newtheorem{Lemma}[Theorem]{Lemma}
\newtheorem{Proposition}[Theorem]{Proposition}

\theoremstyle{definition}

\newtheorem{Example}[Theorem]{Example}
\newtheorem{Remark}[Theorem]{Remark}

\usepackage{amscd,mathtools,bm}
\usepackage[all,cmtip]{xy}
\usepackage{tikz}
\usepackage{tikz-cd}
\usepackage{enumitem}
\usepackage{multirow}

\newcommand{\Section}[1]{\hyperref[sec:#1]{Section~\ref*{sec:#1}}}
\newcommand{\Table}[1]{\hyperref[tab:#1]{Table~\ref*{tab:#1}}}
\newcommand{\eqn}[1]{\hyperref[eqn:#1]{(\ref*{eqn:#1})}}
\newcommand{\Figure}[1]{\hyperref[fig:#1]{Figure~\ref*{fig:#1}}}

\newcommand{\dbb}[1]{[\![#1]\!]}
\newcommand{\br}[1]{\left(#1\right)}

\makeatletter
\def\namedlabel#1#2{\begingroup
 #2%
 \def\@currentlabel{#2}%
 \phantomsection\label{#1}\endgroup
}
\makeatother

\DeclareMathOperator{\GL}{GL}
\DeclareMathOperator{\SL}{SL}

\DeclareMathOperator{\Gr}{Gr}
\DeclareMathOperator{\Fl}{Fl}

\DeclareMathOperator{\QK}{QK}
\DeclareMathOperator{\K}{K}

\DeclareMathOperator{\HH}{H}
\DeclareMathOperator{\pt}{pt}

\def\poly{{\mathrm{poly}}}
\def\loc{{\mathrm{loc}}}

\DeclareMathOperator{\ch}{\ch}

\usepackage{bm}

\DeclarePairedDelimiter{\angles}{\langle}{\rangle}

\newcommand{\bfr}{{\mathbf r}}
\newcommand{\bP}{{\mathbb P}}

\newcommand{\C}{{\mathbb C}}

\newcommand{\Q}{{\mathbb Q}}
\newcommand{\Z}{{\mathbb Z}}

\newcommand{\cO}{{\mathcal O}}

\newcommand{\cQ}{{\mathcal Q}}

\newcommand{\cS}{{\mathcal S}}

\newcommand{\gw}[2]{\angles{ #1 }^{\mbox{}}_{#2}}

\newcommand{\euler}[1]{\chi_{_{#1}}}
\newcommand{\id}{\text{id}}

\DeclareMathOperator{\ev}{ev}

\newcommand{\wb}{\overline}

\newcommand{\ignore}[1]{}

\newcommand{\Mb}{\wb{\mathcal M}}

\begin{document}

\allowdisplaybreaks

\newcommand{\arXivNumber}{2504.07412}

\renewcommand{\PaperNumber}{098}

\FirstPageHeading

\ShortArticleName{Toda-Type Presentations for the Quantum K Theory of Partial Flag Varieties}

\ArticleName{Toda-Type Presentations for the Quantum K Theory\\ of Partial Flag Varieties}

\Author{Kamyar AMINI~$^{\rm a}$, Irit HUQ-KURUVILLA~$^{\rm b}$, Leonardo C.~MIHALCEA~$^{\rm a}$, Daniel ORR~$^{\rm a}$\newline and Weihong XU~$^{\rm c}$}

\AuthorNameForHeading{K.~Amini, I.~Huq-Kuruvilla, L.C.~Mihalcea, D.~Orr and W.~Xu}

\Address{$^{\rm a)}$~Department of Mathematics, Virginia Tech, Blacksburg, VA 24061, USA}
\EmailD{\mail{kamini@vt.edu}, \mail{lmihalce@vt.edu}, \mail{dorr@vt.edu}}

\Address{$^{\rm b)}$~Institute of Mathematics, Academia Sinica, 6F, Astronomy-Mathematics Building,\\
\hphantom{$^{\rm b)}$}~No. 1, Sec. 4, Roosevelt Road, Da-an, Taipei 106319, Taiwan}
\EmailD{\mail{irithk@as.edu.tw}}

\Address{$^{\rm c)}$~Division of Physics, Mathematics, and Astronomy, Caltech, 1200 E. California Blvd.,\\
\hphantom{$^{\rm c)}$}~Pasadena, CA 91125, USA}
\EmailD{\mail{weihong@caltech.edu}}

\ArticleDates{Received April 16, 2025, in final form November 10, 2025; Published online November 20, 2025}

\Abstract{We prove a determinantal, Toda-type, presentation for the equivariant K theory of a partial flag variety ${\rm Fl}(r_1, \dots, r_k;n)$. The proof relies on pushing forward the Toda presentation obtained by Maeno, Naito and Sagaki for the complete flag variety ${\rm Fl}(n)$, via Kato's ${\rm K}_T({\rm pt})$-algebra homomorphism from the quantum K ring of ${\rm Fl}(n)$ to that of ${\rm Fl}(r_1, \dots, r_k;n)$. Starting instead from the Whitney presentation for ${\rm Fl}(n)$, we show that the same pushforward technique gives a recursive formula for polynomial representatives of quantum K Schubert classes in any partial flag variety which do not depend on quantum parameters. In an appendix, we include another proof of the Toda presentation for the equivariant quantum K ring of ${\rm Fl}(n)$, following Anderson, Chen, and Tseng, which is based on the fact that the ${\rm K}$-theoretic $J$-function is an eigenfunction of the finite difference Toda Hamiltonians.}

\Keywords{quantum K theory; partial flag varieties; Toda lattice}

\Classification{14M15; 14N35; 37K10; 05E05}

\section{Introduction}

Let $\Fl(n)$ denote the variety of complete flags in $\C^n$, and let
$\Fl(\bfr,n)=\Fl(r_1, \dots, r_k; n)$ be the variety of partial flags. These are homogeneous under the group
$\SL_n(\C)$, and the restriction of this action to the maximal torus
$T \subset \SL_n(\C)$ has finitely many fixed points, indexed
by a~quotient of the symmetric group $S_n$.
Denote
by $\QK_T(\Fl(\bfr,n))$ the (equivariant, small) quantum K
ring associated to these varieties. This is an algebra over
$\K_T(\pt)[\![Q_1, \dots, Q_k]\!]$, and it has
a~$\K_T(\pt)[\![Q_1, \dots, Q_k]\!]$-basis given by Schubert classes
$\cO^w$ indexed by the torus fixed points. The quantum K multiplication
was defined by Givental and Lee~\cite{givental:onwdvv,lee:QK} in
terms of $3$-point, genus~$0$, $\K$-theoretic Gromov--Witten (KGW) invariants.
Denote by
\[ 0=\cS_0 \subset \cS_1 \subset \dots \subset \cS_k \subset \cS_{k+1}=\C^n \]
the sequence of tautological bundles in $\Fl(r_1, \dots, r_k; n)$; thus
$\mathrm{rank}(\cS_i)=r_i$ for $0\leq i\leq k+1$ with $r_0=0$ and $r_{k+1}=n$.

While the computational foundations of the quantum K rings
of (cominuscule) Grassmannians have been studied for some
time now (see, e.g.,~\cite{BCMP:qkpos,BCMP:qkchev,buch.m:qk,chaput.perrin:rationality,Gorbounov:2014,sinha2024quantum}),
it is only in the last few years that advances have been made in our understanding of quantum K rings for other
flag varieties; see, e.g.,~\cite{anderson2022finite,gu2023quantum,huq2024quantum,kouno.et.al:parabolic,kouno.naito:C,lenart.naito.sagaki:semiinfinite,maeno.naito.sagaki:QKideal,maeno2023presentation}. Many of these advances rely on the groundbreaking
works by
Kato~\cite{kato2019quantum, kato:loop},
who proved the
K-theoretic
version of Peterson's `quantum=affine' statement~\mbox{\cite{IIM:peterson, li.lam.mihalcea.shimozono:qkaffine}}, relating
the quantum K ring of a full flag variety (for an arbitrary complex group $G$) to the
K-homology of the corresponding affine Grassmannian; see also~\cite{chow2022quantum}.
In particular, thanks to results in~\cite{maeno.naito.sagaki:QKideal,maeno2023presentation}
(proving conjectures in~\cite{lenart.maeno:quantum}),
there are now presentations of the quantum K rings by generators and relations
for \smash{$\QK_T(\Fl(n))$}, and we have polynomial representatives (the quantum double Grothendieck polynomials) for Schubert classes.
The generating set of the presentation in~\cite{maeno.naito.sagaki:QKideal}
is in terms of the quantum quotients $\det \cS_{i} / \det \cS_{i-1}$.
We rewrite this presentation in determinantal form
in Theorem~\ref{thm:toda} below. This makes it easier to identify
it with the {\em Toda presentation},
which is obtained by taking symbols of the finite difference
Toda operators studied by Givental and Lee~\cite{givental_lee},
and also by Anderson, Chen and Tseng in~\cite{act:2017}, see also
\cite{koroteev} and Appendix~\ref{sec:toda} below.

Our main result is to generalize the Toda presentation from $\QK_T(\Fl(n))$
to one for the ring
$\QK_T(\Fl(\bfr,n))$ associated to partial flag varieties. To state it, let
\[
Y^{(j)}=\bigl(Y^{(j)}_1,\dots,{ Y^{(j)}_{r_{j+1}-r_j}}\bigr), \qquad 0\leq j\leq k
\]
be formal variables and $e_\ell$ be the $\ell$-th elementary symmetric polynomial. Let $T_1,\dots, T_n \in \K_T(\pt)$ be given by
the decomposition of $\C^n$ into one dimensional $T$-modules, that is, $\wedge^\ell(\C^n)=e_\ell(T_1,\dots,T_n)$. To distinguish from multiplication in $K_T(\Fl(\bfr,n))$, we denote the multiplication in $\QK_T(\Fl(\bfr,n))$ by $\star$.

\begin{Theorem}[Theorem~\ref{thm:main}]
 \label{thm:main-intro}
 The ring $\QK_T(\Fl(\bfr,n))$ is isomorphic to $R\dbb{Q}/J_Q$, where
 \[
 R=\K_T(\pt)\bigl[e_1\bigl(Y^{(j)}\bigr),\dots,e_{r_{j+1}-r_j}\bigl(Y^{(j)}\bigr),\, 0\leq j\leq k\bigr],
 \]
 and $J_Q\subset R[\![Q]\!]=R[\![Q_1,\dots,Q_k]\!]$
is the ideal generated by the coefficients of $y$ in
\begin{align*}
 \prod_{\ell=1}^{n}(1+yT_\ell)-
 &\begin{vmatrix}
 A_0 & B_1 & & & \\
 1 & A_1 & B_2 & & \\ & \ddots & \ddots & \ddots & \\ & & 1 & A_{k-1} & B_{k}\\ & & & 1 & A_{k}
 \end{vmatrix}^\star,
 \end{align*}
 where \[A_j=\prod_{\ell=1}^{r_{j+1}-r_j}\bigl(1+y Y^{(j)}_\ell\bigr)+B_j,\qquad B_j=y^{{r_{j+1}-r_j}}\frac{Q_j}{1-Q_j}\prod_{\ell=1}^{r_{j+1}-r_j}Y^{(j)}_\ell,\]
 with the convention that $Q_0=0$.

 More precisely,
there exists a $\K_T(\pt)[\![Q]\!]$-algebra isomorphism
\[\Psi\colon \ R[\![Q]\!]/J_Q \to \QK_T(\Fl(r_1,\dots,r_k)), \qquad e_\ell\bigl(Y^{(j)}\bigr) \mapsto \wedge^\ell (\cS_{j+1}/\cS_j) \]
for $j=0,\dots,k$ and $\ell=1,\dots,r_{j+1}-r_{j}$.
\end{Theorem}

Our proof is by decreasing induction on $k$. The initial case, $k=n-1$, is the main result of~\cite{maeno.naito.sagaki:QKideal}, rewritten in determinantal form in Theorem~\ref{thm:toda} below. For the induction step, we use a~result of Kato~\cite{kato2019quantum} which states that there is a
$\K_T(\pt)$-{\em algebra} homomorphism
\begin{gather}
\QK_T(\Fl(n)) \to \QK_T(\Fl(r_1, \dots , r_k;n)); \nonumber\\
\qquad{}\cO^w \mapsto \pi_*(\cO^w) ,
\qquad Q_r \mapsto \begin{cases} 1, & r \notin \{ r_1, \dots, r_k \}, \\
Q_i, & r=r_i ,\end{cases}\label{E:kato-push}
\end{gather}
which extends the usual projection map $\pi_*\colon \K_T(\Fl(n)) \to \K_T(\Fl(\bfr,n))$.
Note that the classical~$\pi_*$ is {\em not} a ring map. (Kato's result is for general complex, simple groups $G$.) We use this to show that the ideal $J_Q$ is contained in the ideal of relations.
For the specialization ${Q_i \mapsto 1}$ to be well defined, one needs to work with
polynomials in $Q_1, \dots, Q_{n-1}$; see Section~\ref{sec:prelim}.
Pushing forward the original Toda relations is not possible, due to poles at $Q_i=1$.
We had to rewrite these relations, and additionally
use an extra identity due to Maeno, Naito, and Sagaki
(cf.~Proposition~\ref{prop:MNS} below), in order for the push forward to be performed.
Finally, it follows from~\cite{GMSXZZ:Nakayama} that the ideal $J_Q$ coincides with the ideal of relations.

The same pushforward technique may be applied to the {\em Whitney presentation},
conjectured in~\cite{gu2023quantum,GMSXZZ:QKW}, and for which a proof was recently announced in~\cite{huq2024quantum}; see also~\cite{Gu:2020zpg,gu2022quantum} for the Grassmannian case.
This is a presentation for $\QK_T(\Fl(\bfr,n))$ with generators $\wedge^k(\cS_i)$ and $\wedge^\ell(
 \cS_i/\cS_{i-1}
)$. We prove in Proposition~\ref{prop:whit-toda} that if one eliminates the variables corresponding to
classes $\wedge^k(\cS_i)$ in the Whitney presentation,
then one recovers the Toda presentation.

Our methods also provide a different proof of the Whitney
presentation for $\QK_T(\Fl(\bfr,n))$, once the Whitney
presentation for $\QK_T(\Fl(n))$ (a special case of results from~\cite{huq2024quantum})
is assumed; see Remark~\ref{rmk:QKW}. The details of this proof are omitted, as
they follow closely the proof of Theorem~\ref{thm:main-intro}.

In a further application of our technique, using the aforementioned
Whitney presentation, we rewrite the formula
from~\cite{maeno2023presentation}
of the quantum double Grothendieck
polynomial of the class of a~point in $\Fl(n)$~\cite{maeno2023presentation} in terms
of the classes $\lambda_y(\cS_i)$.
Surprisingly, the resulting class is {\em independent}
of the quantum parameters $Q_i$.
Pushing forward this class results in a polynomial representative
for the class of the (Schubert) point in any $\QK_T(\Fl(\bfr,n))$
which is independent of $Q_i$. The outcome is the following.

\begin{Theorem}[Theorem~\ref{thm:w0}]
Let $\cO^{w_0} \in \QK_T(\Fl(r_1,\dots, r_k;n))$ be the class of the Schubert point.
Then the following holds:
 \begin{gather*}
 \cO^{w_0}=\sideset{}{^\star}\prod_{i=1}^{k}\sideset{}{^\star}\prod_{j=r_i}^{r_{i+1}-1}\lambda_{-1}({\rm e}^{-\epsilon_{n-j}}\cS_i) ,
 \end{gather*}
where ${\rm e}^{\epsilon_i} \in \K_T(\pt)$ denotes the $($class of the$)$ $1$-dimensional $T$-representation
with weight $\epsilon_i$.
\end{Theorem}

In the usual (equivariant) K theory of $\Fl(n)$ this follows from Fulton's results in~\cite{fulton:flags}
showing that the Schubert point $X^{w_0}$ is the zero locus of a section of a vector bundle;
see also~\cite[Theorem~3]{fulton.lascoux}.
Using the left divided difference
operators in $\QK_T(\Fl(\bfr,n))$ defined in~\cite{Mihalcea2022Left},
this results in a~recursive formula for any Schubert class, giving
polynomial representatives in terms of exterior
powers $\wedge^i \cS_j$ which do not depend on quantum parameters.
See Theorem~\ref{thm:double-Grothendieck}.
Precursors of this `quantum=classical' phenomenon for polynomial representatives
of quantum Schubert classes have been observed for (isotropic) Grassmannians
\cite{bertram:quantum,BCFF,Gorbounov:2014,IMN:factorial,mihalcea:giambelli}, but to our knowledge this is new for
(partial) flag varieties. Recently, we learned that T.~Kouno found a~similar phenomenon
in the quantum K ring of the symplectic flag varieties $\mathrm{Sp}_{2n}/B$.

In Appendix~\ref{sec:toda}, we follow Anderson, Chen, and Tseng's
treatment in the unpublished
note~\cite{act:2017} to give another proof
of the Toda presentation for $\QK_T(\Fl(n))$,
independent of the one from~\cite{maeno.naito.sagaki:QKideal}.
The proof combines results of
Givental and Lee~\cite{givental_lee}, which
states that the K-theoretic J-function of $\Fl(n)$ is an eigenfunction of
the first (finite difference) Toda Hamiltonian, with results of Iritani, Milanov and Tonita
\cite{Iritani:2013qka}, which relates this fact to relations in the quantum
K theory ring. We do not claim any originality in this argument,
but we found it valuable to include it here, as it puts together results
from the followup papers~\cite{anderson2022finite} and~\cite{kato:loop}; see especially
Proposition~\ref{thm:ACTI-Kato}.

\subsection{A logical roadmap} There are several recent results in the literature which inform the Toda and Whitney
presentations we prove in this paper.
Since we do not attempt to give self-contained
proofs, we provide next a roadmap of the logical
implications we rely on, which a reader may find useful.

Our proofs of both the Toda presentation, and the Whitney presentation,
for $\QK_T(\Fl(\bfr,n))$, from Theorem~\ref{thm:main} (resp.\ \eqref{E:QKWhitney}, see Remark~\ref{rmk:QKW})
rely on the following: the Toda (resp.\ Whitney) presentation for $\QK_T(\Fl(n))$;
Kato's push forward homomorphism from~\eqref{E:kato-push}; and the key
technical result from Proposition~\ref{prop:MNS}, proved in~\cite{maeno.naito.sagaki:QKideal},
which in turn relies on Kato's work~\cite{kato:loop}.

There are two proofs of the Toda presentation for $\QK_T(\Fl(n))$, one in
\cite{maeno.naito.sagaki:QKideal}, relying on~\cite{kato:loop}, and another
which may be deduced from~\cite{act:2017}, relying on results from
\cite{givental_lee} and~\cite{Iritani:2013qka};~cf.~Appendix~\ref{sec:toda}.

There are also two proofs of the Whitney presentation for $\QK_T(\Fl(n))$.
One was recently announced by Huq-Kuruvilla
\cite{huq2024quantum} (for all rings $\QK_T(\Fl(\bfr,n))$), and
it uses the technique of abelian-nonabelian correspondence, independent
of Kato's results. Another proof of the Whitney presentation for
 $\QK_T(\Fl(n))$ is given in~\cite[Section~6]{GMSXZZ:QKW}.
It relies on the recent proof of the quantum~K divisor axiom~\cite{LNSX:QKdiv},
and ultimately on Kato's results.

\section{Preliminaries}
\subsection{Equivariant K theory of Grassman bundles}
Let $T$ be a linear algebraic group. For any projective $T$-variety $Z$,
let $\K_T(Z)$ be the equivariant $\K$-theory
ring, defined as the Grothendieck ring of $T$-equivariant algebraic vector
bundles. This ring is an algebra over $\K_T(\pt)$, the representation ring of
$T$. Let $\euler{Z}\colon \K_T(Z) \to \K_T(\pt)$ be the pushforward map along the
structure morphism.

For $E \to Z$ a $T$-equivariant
vector bundle of rank $\operatorname{rk}E$, we denote by
\[ \lambda_y(E) := 1 + y [E] + \dots + y^{\operatorname{rk}E} \bigl[\wedge^{\operatorname{rk}E} E\bigr] \in \K_T(Z)[y] \]
the Hirzerbruch $\lambda_y$ class of $E$. This class is multiplicative for short exact sequences. {In an abuse of notation, we often write $E$ for the class $[E]$ in $\K_T(Z)$.} Note that for a rank $e$ equivariant vector bundle $E$, and a character ${\rm e}^\chi \in \K_T(\pt)$,
\[ \lambda_y( {\rm e}^\chi \otimes E) = \lambda_{y {\rm e}^{\chi}}(E) = \sum_{i=0}^{e} y^i {\rm e}^{i \chi} \otimes \wedge^i E . \]
As is customary, we will often remove the $\otimes$ symbol from the notation.

Denote by
$\pi\colon \mathbb{G}(r,E) \to Z$ the Grassmann bundle over $Z$.
It is equipped with a tautological sequence
$0 \to \underline{\cS} \to \pi^* E \to \underline{\cQ} \to 0$ over
$\mathbb{G}(r,E)$. The following result follows
from~\cite[Proposition~2.2]{kapranov:Gr}, see also~\cite[Proposition~3.2 and Corollary~3.3]{gu2022quantum}.
(Kapranov proved this when $Z=\pt$; the relative version follows immediately
using that $\pi$ is a $T$-equivariant locally trivial fibration.) We only state the special cases that will be used in this paper. See the above references for the full generality.

\begin{Proposition}[Kapranov]\label{prop:relativeBWB}
 There are the following isomorphisms of $T$-equivariant vector bundles:
 \begin{enumerate}\itemsep=0pt
 \item[$(1)$] for all $i \ge 0$, $\ell>0$ the higher direct images, $R^i \pi_* \bigl(\wedge^\ell\underline{\cS}\bigr) = 0$;
 \item[$(2)$] for all $\ell\geq 0$, \[ R^i \pi_*\bigl(\wedge^\ell\underline{\cQ}\bigr) = \begin{cases} \wedge^\ell E, & i=0, \\ 0, & i>0 . \end{cases}\]
 \end{enumerate}
 \end{Proposition}

\subsection{(Equivariant) quantum K theory of flag varieties}\label{sec:prelim}
Let $\mathbf{r}= (r_1, \dots, r_k)$. We consider \[X=\Fl(\bfr,n),\] which parametrizes flags of vector spaces
$F_{1} \subset F_{2} \subset \dots \subset F_{k} \subset \C^n$
with $\dim F_{i} = r_i$ for~${1 \le i \le k}$.

Let $M_{d,n} := \Mb_{0,n}(X,d)$ be the moduli space of genus zero degree
$d$ stable maps to $X$ with~$n$ marked points. Given classes
$a_1, \dots, a_n \in \K_T(X)$,
define the $\K$-theoretic Gromov--Witten invariants by
\[ \langle a_1, \dots, a_n \rangle_{d} =
\euler{M_{d,n}}\Biggl(\prod_{i = 1}^n\ev_i^*(a_i)\Biggr)\in \K_T(\pt) . \]
Non-equivariant Gromov--Witten invariants are obtained by replacing $T$ with the
trivial group; these Gromov--Witten invariants are integers.

For $d=(d_1,\dots,d_k)\in\HH_2(X,\Z)\cong\Z^k$, we write $Q^d$ for \smash{$\prod_{i=1}^k Q_i^{d_i}$}. Here $Q_i$ corresponds to the Poincar{\'e} dual of the first Chern class $-c_1(\det \cS_i)$. Following~\cite{givental:onwdvv,lee:QK}, the $T$-equivariant (small) quantum
K theory ring is
\begin{equation*}
\QK_T(X) = \K_T(X) \otimes_{\K_T(\pt)} \K_T(\pt)[\![Q]\!]
\end{equation*}
as a $\K_T(\pt)[\![Q]\!]$-module.
It is equipped with a commutative, associative
product, denoted by $\star$, which is
determined by the condition
\begin{equation*}
 (\!(\sigma_1\star\sigma_2,\sigma_3)\!)=\sum_d
 Q^d \gw{\sigma_1,\sigma_2,\sigma_3}{d}\qquad\text{for all $\sigma_1,\sigma_2,\sigma_3\in \K_T(X)$},
\end{equation*}
where
\[
 (\!(\sigma_1,\sigma_2)\!)\coloneqq \sum_{d}Q^d \gw{\sigma_1,\sigma_2}{d}
\]
is the quantum $\K$-metric.

It was proved in~\cite{anderson2022finite, kato:loop} that for $\sigma_1,\sigma_2\in\K_T(X)$, the product $\sigma_1\star\sigma_2$ can always be expressed as a polynomial in $Q$ with coefficients in $\K_T(X)$. It follows that
\[
 \QK_T^\poly(X)\coloneqq\K_T(X) \otimes_{\K_T(\pt)}\K_T(\pt)[Q]
\]
is a subring of $\QK_T(X)$.

Let $Y=\Fl\bigl(r_1,\dots, \widehat{r}_i,\dots, r_k;n\bigr)$ and $\pi\colon X\to Y$ be the natural map. Let also $\widehat{\mathbf{r}} = \bigl(r_1, \dots,\allowbreak \widehat{r}_i,\dots, r_k\bigr)$. The following theorem is a specialization of results proved in~\cite{kato2019quantum}.

\begin{Theorem}[Kato]\label{thm:kato-morphism}
 There is a surjective ring homomorphism
 \[
 \Phi\colon\ \QK^\poly_T(X)\to\QK^\poly_T(Y)
 \]
 given by $\sigma\mapsto{\pi}_*\sigma$ for all $\sigma\in\K_T(X)$ and
 \[
 Q_j\mapsto\begin{cases}
 Q_j, & j\neq i,\\
 1, & j=i
 \end{cases} \qquad \text{for $1\leq i\leq k$}.\]
\end{Theorem}
It follows from Theorem~\ref{thm:kato-morphism} that
Kato's homomorphism extends naturally to
\begin{equation}\label{E:Kato-loc-hom}
\Phi\colon \QK_T^{\loc(\mathbf{\hat r})}(X) \to
\QK_T^{\loc(\mathbf{\hat r})}(Y) ,
\end{equation}
where $\loc(\mathbf{\hat r})$ indicates localization at the multiplicative set generated by
$1- Q_j$ for $j \neq i$.

\subsection[The Toda presentation for Fl(n)]{The Toda presentation for $\boldsymbol{\Fl(n)}$}

The variety $\Fl(n)=\Fl(1,\dots,n-1;n)$ is equipped with tautological vector bundles
\[
 0=\cS_{0}\subset\cS_{1}\subset\dots \subset \cS_{n-1}\subset \cS_{n}=\C^n ,
\]
where $\cS_j$ has rank $r_j$. It can also be viewed as $\SL_n/B$, where $B\subset \SL_n$ is a Borel subgroup. Let~${T\subseteq B}$ be a maximal torus in $\SL_n$.

The following is the main result of~\cite{maeno.naito.sagaki:QKideal} (see Remark~\ref{rmk:mns} for more details). The relation~\eqref{eqn:toda} can also be recovered from the connection between the $J$-function of the full flag variety and the relativistic Toda lattice established by Givental and Lee in~\cite{givental_lee}. This observation was made in the unpublished note~\cite{act:2017} of Anderson--Chen--Tseng, but removed from the published version of their paper. For the sake of completeness, we give a brief account in Appendix~\ref{sec:toda}.

\begin{Theorem}\label{thm:toda}
The ring $\QK_T(\Fl(n))$ is isomorphic to $R'[\![Q]\!]/J'_Q$, where $R'$ is equal to the ring \smash{$\K_T(\pt)\bigl[P_1^{\pm },\dots,P_{n}^\pm\bigr]$} and the ideal $J'_Q\subset R'\dbb{Q}=R'\dbb{Q_1,\dots,Q_{n-1}}$ is generated by the coefficients $y$ in
 \begin{gather}\label{eqn:toda}
 \lambda_y(\C^n) -
 \begin{vmatrix}
 1+y\frac{P_1}{P_0} & y\frac{P_2}{P_1}Q_1 & & & & \\[0.9mm]
 1 & 1+y\frac{P_2}{P_1} & y\frac{P_3}{P_2}Q_2 & & & \\[0.9mm]
 & 1 & 1+y\frac{P_3}{P_2} & y\frac{P_{4}}{P_{3}}Q_3 & & \\
 & & \ddots & \ddots & \ddots & \\
& & & 1 & 1+y\frac{P_{n-1}}{P_{n-2}} & y\frac{P_n}{P_{n-1}}Q_{n-1}\\[1mm]
& & & & 1 & 1+y\frac{P_n}{P_{n-1}}
\end{vmatrix}^\star,
\end{gather}
here $P_0=1$ by convention, and $\lambda_y(\C^n)\in\K_T(\pt)[y]$.

More precisely,
there exists a $\K_T(\pt)[\![Q]\!]$-algebra isomorphism $\Psi'\colon R'[\![Q]\!]/J'_Q \to \QK_T(\Fl(n))$ that sends $P_j$ to $\det \cS_j$ for all $j=1,\dots,n$.
\end{Theorem}

\begin{Remark} \label{rmk:mns}
 Theorem~\ref{thm:toda} is proved in~\cite{maeno.naito.sagaki:QKideal} using results of Kato~\cite{kato:loop} based on the semi-infinite
 flag variety.
 		The connection between our statement of Theorem~\ref{thm:toda} and that of~\cite{maeno.naito.sagaki:QKideal} is seen as follows. Define the \textit{Toda polynomials} \smash{$T^{(n)}_k$} for $k=1,\dots,n$ by
		\begin{align*}
		T^{(n)}_k = \sum_{0= i_0<\dots<i_k\leq n}\ \prod_{s=1}^{k}\frac{P_{i_s}}{P_{i_{s}-1}}\left(1-Q_{i_s-1}\right)^{1-\delta_{i_s-i_{s-1},1}}.
		\end{align*}
		These elements of $\Z\bigl[P_1^{\pm},\dots,P_{n}^{\pm}\bigr][\![Q]\!]$ (where $P_0=1$ and $Q_0=0$ by convention) are symbols of the finite-difference Toda Hamiltonians~\cite{etingof} (see also~\cite{act:2017,gloI,givental_lee,koroteev}).
		
		Letting \smash{$T^{(n)}=\sum_{k=0}^n T^{(n)}_k y^k$} where \smash{$T^{(n)}_0=1$}, we claim that $T^{(n)}(y)$ is equal to the determinant of the matrix appearing in the Toda relations~\eqref{eqn:toda}, namely,
		\begin{align*}
 T^{(n)}(y)=\begin{vmatrix}
 1+y\frac{P_1}{P_0} & y\frac{P_2}{P_1}Q_1 & & & & \\[0.9mm]
 1 & 1+y\frac{P_2}{P_1} & y\frac{P_3}{P_2}Q_2 & & & \\[0.9mm]
 & 1 & 1+y\frac{P_3}{P_2} & y\frac{P_{4}}{P_{3}}Q_3 & & \\
 & & \ddots & \ddots & \ddots & \\
& & & 1 & 1+y\frac{P_{n-1}}{P_{n-2}} & y\frac{P_n}{P_{n-1}}Q_{n-1}\\[1mm]
& & & & 1 & 1+y\frac{P_n}{P_{n-1}}
\end{vmatrix}^\star.
		\end{align*}
		This is verified by showing that $T^{(n)}$ satisfies the recursion
		\begin{equation*}
T^{(n)}= T^{(n-1)} \biggl(1+ y \frac{P_n}{P_{n-1}}\biggr) - y Q_{n-1} \frac{P_n}{P_{n-1}} T^{(n-2)} ,
		\end{equation*}
		and then applying Lemma~\ref{L:r-det} below (with $n$ playing the role of $j$ there).
		\end{Remark}

 \begin{Remark}
 Upon the specialization $Q_1=\dotsm=Q_{n-1}=0$, the Toda presentation $R'[\![Q]\!]/J'_Q\cong\QK_T(\Fl(n))$ becomes the Borel presentation
 \[
 \K_T(\pt)\bigl[P_1^\pm,\dots,P_n^{\pm}\bigr]/J\cong \K_T(\Fl(n)),
 \]
 where $J$ is the ideal generated by the coefficients of $y$ in
 \[
 \lambda_y(\C^n)-\sideset{}{^\star}\prod_{j=0}^{n-1} (1+yP_{j+1}/P_j)
 \]
 and~$P_j$ corresponds to $\det \cS_j$ for all $j=1,\dots,n$.
 \end{Remark}
 \begin{Lemma}
 \label{L:r-det}
 Suppose $U_j$ for $0\le j\le k+1$ and $A_j$, $B_j$ for $0\le j\le k$ are elements of a~commutative ring with $1$ such that the $U_j$ satisfy the recursion
 \begin{equation*}
 U_{j+1}=A_jU_{j}-B_jU_{j-1},\qquad 0\le j\le k
 \end{equation*}
 with initial conditions $U_0=1$, $U_{-1}=0$. Then, for all $0\le j\le k+1$, one has
 \begin{equation}
 \label{E:r-det}
 U_j=
 \begin{vmatrix}
 A_0 & B_1 & & & \\
 1 & A_1 & B_2 & & \\
 & \ddots & \ddots & \ddots & \\
 & & 1 & A_{j-2} & B_{j-1}\\
 & & & 1 & A_{j-1}
 \end{vmatrix}.
 \end{equation}
 \end{Lemma}

 \begin{proof}
 One simply expands along the last row or column to see that the determinant in~\eqref{E:r-det} satisfies the recursion. Observe that the initial values $U_0=1$ and $U_1=A_0$ agree. This completes the proof.
 \end{proof}

 Before finishing this section, we record the following, which follows from~\cite[Proposition~5.2]{maeno.naito.sagaki:QKideal}.

 \begin{Proposition}[Maeno--Naito--Sagaki]\label{prop:MNS}
 In $\QK_T(\Fl(n))$, the following relations hold:
 \begin{gather*}
 \det \cS_i \star \det \cS_j/\cS_i = (1-Q_i) \det \cS_j, \qquad 1 \le i < j \le n.
 \end{gather*}
 \end{Proposition}

\section[Toda-type presentations for the equivariant quantum K theory of partial flag varieties]{Toda-type presentations for the equivariant quantum\\ K theory of partial flag varieties}

To begin, we observe that the Toda presentation in Theorem~\ref{thm:toda} can be rewritten as follows.
\begin{Corollary}\label{cor:fln}
 The ring $\QK_T(\Fl(n))$ is isomorphic to $R\dbb{Q}/J_Q$, where
 \[
 R=\K_T(\pt)\bigl[Y^{(0)}, \dots, Y^{(n-1)}\bigr],
 \]
 and $J_Q\subset R\dbb{Q}=R\dbb{Q_1,\dots,Q_{n-1}}$ is generated by the coefficients of $y$ in
 \begin{gather*}
 \lambda_y(\C^n) - \\
 \begin{vmatrix}
 1+yY^{(0)}\frac{1}{1-Q_0} & yY^{(1)}\frac{Q_1}{1-Q_1} & & & & \\[1mm]
 1 & 1+yY^{(1)}\frac{1}{1-Q_1} & yY^{(2)}\frac{Q_2}{1-Q_2} & & & \\ 
& \ddots & \ddots & \ddots \\ & & 1 & 1+yY^{(n-2)}\frac{1}{1-Q_{n-2}} & yY^{(n-1)}\frac{Q_{n-1}}{1-Q_{n-1}}\\[1mm] & & & 1 & 1+yY^{(n-1)}\frac{1}{1-Q_{n-1}}
\!\!\!\end{vmatrix}^\star
\end{gather*}
with the convention that $Q_0=0$.

More precisely,
there exists a $\K_T(\pt)[\![Q]\!]$-algebra isomorphism $\Psi\colon R[\![Q]\!]/J_Q \to \QK_T(\Fl(n))$ that sends $Y^{(j)}$ to $\cS_{j+1}/\cS_j$ for $j=1,\dots,n-1$.
\end{Corollary}
\begin{proof}
 Identifying $P_{j+1}/P_j$ with $Y^{(j)}/(1-Q_j)$ gives an isomorphism between $R[\![Q]\!]/J_Q$ and $R'[\![Q]\!]/J'_Q$. More precisely, define a $\K_T(\pt)[\![Q]\!]$ homomorphism $\Phi\colon R'[\![Q]\!]/J'_Q\to R[\![Q]\!]/J_Q$ by
\[
 \Phi(P_j)=\prod_{i=0}^{j-1}\frac{Y^{(i)}}{1-Q_i}, \qquad 1\leq j\leq n-1.
\]
Note that in $R[\![Q]\!]/J_Q$, we have \smash{$\det\C^{n}=\prod_{i=0}^{n-1}Y^{(i)}/(1-Q_i)$}, which implies all $Y^{(j)}$ are invertible. Since the relations match, the homomorphism $\Psi$ is well-defined and injective. Since \smash{$(1-Q_j)P_{j+1}/P_j$} is sent to $Y^{(j)}$ for $0\leq j\leq n-1$, it is also surjective. Finally, the~geometric interpretation follows from Proposition~\ref{prop:MNS}.
\end{proof}

Next, we generalize Corollary~\ref{cor:fln} to all partial flag varieties utilizing Theorem~\ref{thm:kato-morphism}, Proposition~\ref{prop:MNS}, and the Nakayama-type result from~\cite{GMSXZZ:QKW, gu2022quantum}.

\begin{Theorem}\label{thm:whit-short}
 In $\QK_T(\Fl(\bfr,n))[y]$, the following relation hold:
 \begin{align}\label{eqn:whit-short}
 \lambda_y(\C^n)-
 &\begin{vmatrix}
 A_0 & B_1 & & & \\
 1 & A_1 & B_2 & & \\ & \ddots & \ddots & \ddots & \\ & & 1 & A_{k-1} & B_{k}\\ & & & 1 & A_{k}
 \end{vmatrix}^\star,
 \end{align}
 where \[
 B_j=y^{r_{j+1}-r_j}\frac{Q_j}{1-Q_j}\det(\cS_{j+1}/\cS_j),\qquad
 A_j=\lambda_y(\cS_{j+1}/\cS_j)+B_j.
\]
\end{Theorem}
\begin{proof}
Let $X=\Fl(r_1,\dots,r_k;n)$, $Y=\Fl\bigl(r_1,\dots, \widehat{r}_i,\dots, r_k;n\bigr)$, and $\pi\colon X\to Y$ be the natural map. Let
\[
 0=\cS_0\subset\cS_1\subset\dots\subset\cS_k\subset\cS_{k+1}=\C^n
\]
be the sequence of tautological bundles on $X$. Note that all but $\cS_i$ are pulled back from $Y$. With a slight abuse of notation, we denote the sequence of tautological bundles on $Y$ by
\[
 0=\cS_0\subset\cS_1\subset\dots\cS_{i-1}\subset\cS_{i+1}\subset\dots\subset\cS_k\subset\cS_{k+1}=\C^n.
\]
Note that the elements $B_1,\dots,B_{i-2},B_{i+1},\dots,B_k$ as well as $A_1.\dots, A_{i-2},A_{i+1},\dots,A_k$ in the ring $\QK_T(X)[y]$ stay the same under pushforward along $\pi$. By a slight abuse of notation, we~also think of them as elements of $\QK_T(Y)[y]$.

By induction, we assume that relation~\eqref{eqn:whit-short} holds for $X$, i.e.,
\begin{gather}\label{eqn:whitX}
 \lambda_y(\C^n)-\begin{vmatrix}
 A_0 & B_1 & & & & & & & & & \\
 1 & A_1 &B_2 & & & & & & & & \\
 & \ddots & \ddots & \ddots & & & & & & & \\
 & & \ddots & \ddots & B_{i-2} & & &\\ & & & 1 & A_{i-2} & B_{i-1} & & & & &\\ & & & & 1 & A_{i-1} & B_i & & & &\\
 & & & & & 1 & A_i & B_{i+1} & & &\\
 & & & & & & 1 & A_{i+1} & \ddots & &
 \\
 & & & & & & & \ddots & \ddots & \ddots & \\
 & & & & & & & & 1 & A_{k-1} & B_{k}\\
 & & & & & & & & & 1 & A_{k}
 \end{vmatrix}^\star
\end{gather}
holds in \smash{$\QK_T^{\loc(\mathbf{r})}(X)[y]$}
for $1 \le j \le k$, and we will show that the (localized) Kato's pushforward~\eqref{E:Kato-loc-hom} of this relation gives relation~\eqref{eqn:whit-short} on $Y$.

Relation~\eqref{eqn:whit-short} on $Y$ reads
\begin{gather}\label{eqn:whitY}
 \lambda_y(\C^n)-\begin{vmatrix}
 A_0 & B_1 & & & \\
 1 & A_1 & B_2 & & \\ & \ddots & \ddots & \ddots & \\
 & & \ddots & \ddots & B_{i-2}\\ & & & 1 & A_{i-2} & B'_{i-1}\\ & & & & 1 & A'_{i-1} & B_{i+1}\\ & & & & & 1 & A_{i+1} & \ddots
 \\ & & & & & & \ddots & \ddots & \ddots \\ & & & & & & & 1 & A_{k-1} & B_{k}\\ & & & & & & & & 1 & A_{k}
 \end{vmatrix}^\star,
\end{gather}
where
\begin{gather*}
 B_{i-1}'=y^{r_{i+1}-r_{i-1}}\frac{Q_{i-1}}{1-Q_{i-1}}\det\left(\cS_{i+1}/\cS_{i-1}\right),\qquad A_{i-1}'=\lambda_y(\cS_{i+1}/\cS_{i-1})+B'_{i-1},
\end{gather*}
regarded as elements in \smash{$\QK_T^{\loc(\widehat{\mathbf{r}})}(Y)[y]$}.

By the projection formula, to prove~\eqref{eqn:whitY}, it suffices to prove the pushforward along $\pi$ of~\eqref{eqn:whitX} agrees with~\eqref{eqn:whitY}. We compare the two determinants by expanding along columns. Expanding along the column containing $B'_{i-1}$, we have that the determinant in~\eqref{eqn:whitY} is of the form
\begin{gather*}
 -B'_{i-1}\star C'+A'_{i-1}\star D'-E';
\end{gather*}
expanding along the two columns containing $B_{i-1}$ or $B_i$, we have that the determinant in~\eqref{eqn:whitX} is of the form
\begin{gather*}
 \begin{vmatrix}B_{i-1} & 0\\
 A_{i-1} & B_{i}\end{vmatrix}^\star\star0-\begin{vmatrix}
 B_{i-1} & 0\\
 1 & A_{i}
 \end{vmatrix}^\star\star C +\begin{vmatrix}
 B_{i-1} & 0\\
 0 & 1
 \end{vmatrix}^\star\star F\\
 \qquad{} +\begin{vmatrix}
 A_{i-1} & B_i\\1 & A_i
 \end{vmatrix}^\star\star D -\begin{vmatrix}
 A_{i-1} & B_i\\
 0 & 1
 \end{vmatrix}^\star\star E+\begin{vmatrix}
 1 & A_i\\
 0 & 1
 \end{vmatrix}^\star\star 0.
 \end{gather*}
Note that $C$, $D$, $E$, $F$ stay the same under the pushforward, and it is straightforward to check that
\[
C'=C,\qquad D'=D,\qquad E'=E.
\]
The rest follows from Lemma~\ref{lemma:main-push} below.
\end{proof}

\begin{Lemma}\label{lemma:main-push} The following hold:
 \begin{enumerate}[label=$(\alph*)$]\itemsep=0pt
 \item $\pi_* \begin{vmatrix}
 A_{i-1} & B_i\\
 0 & 1
 \end{vmatrix}^\star= 1$;
 \item $\pi_* \begin{vmatrix}
 B_{i-1} & 0\\
 0 & 1
 \end{vmatrix}^\star = 0$;
 \item $\pi_* \begin{vmatrix} B_{i-1} & 0\\ 1 & A_{i} \end{vmatrix}^\star = B_{i-1}'$;
 \item Assume that $r_i - r_{i-1} = 1$. Then
 \[
 \pi_* \begin{vmatrix} A_{i-1} & B_i\\1 & A_i
 \end{vmatrix}^\star = A_{i-1}'.
 \]
 \end{enumerate}
 \end{Lemma}
 \begin{proof}
 Note that $X$ may be realized as the Grassmann bundle $\mathbb{G}(r_i-r_{i-1}, \cS_{i+1}/\cS_{i-1})$
 over~$Y$, with tautological sequence $0 \to \cS_i/\cS_{i-1} \to \cS_{i+1}/\cS_{i-1} \to \cS_{i+1}/\cS_{i} \to 0$.
 It follows from Proposition~\ref{prop:relativeBWB} that
 \begin{equation}\label{eqn:pushforwards-lambda}
 \pi_*(\lambda_y(\cS_{i+1}/\cS_{i}))=\sum_{j=0}^{r_{i+1}-r_i}y^j\wedge^j(\cS_{i+1}/\cS_{i-1}), \qquad\pi_* (\lambda_y(\cS_i/\cS_{i-1})) = 1.
 \end{equation}

 For (a), (b), note that \smash{$A_{i-1}, B_{i-1}\in \QK_T^{\loc(\mathbf{\hat r})}(X)$}, so we may use~\eqref{E:Kato-loc-hom}, and it follows that
 \begin{gather*}
 \pi_*B_{i-1}=0,\qquad \pi_*A_{i-1} = 1.
 \end{gather*}
 Note that by Proposition~\ref{prop:MNS} and Theorem~\ref{thm:kato-morphism}, we have
 \begin{gather}
 \det\cS_{j}\star\det(\cS_{j+1}/\cS_{j})=(1-Q_{j})\det\cS_{j+1}\qquad \text{for $0\leq j\leq k$, in $\QK_T(X)$},\label{eqn:relX}\\
 \det\cS_{i-1}\star\det(\cS_{i+1}/\cS_{i-1})=(1-Q_{i-1})\det\cS_{i+1}\qquad \text{in $\QK_T(Y)$}.\label{eqn:relY}
\end{gather}
 To prove (c), we obtain from definition
 \begin{align}
 \begin{vmatrix}
 B_{i-1} & 0\\
 1 & A_{i}
 \end{vmatrix}^\star
 &{}= B_{i-1}A_i=B_{i-1} \star (\lambda_y(\cS_{i+1}/\cS_i)+B_i)\nonumber\\
 &{}= B_{i-1}\star \lambda_y(\cS_{i+1}/\cS_i)+ B_{i-1} \star B_i .\qquad\label{eqn:3-expand}
 \end{align}
 The element $B_{i}$ cannot be pushed forward, as it contains $1-Q_i$ in the denominator. However,
 we use~\eqref{eqn:relX} to calculate
 \begin{align*}
 B_{i-1} \star B_i &{}= y^{r_{i+1}-r_{i-1}} \frac{Q_{i-1} Q_i}{(1-Q_{i-1})(1-Q_i)} \det (\cS_{i+1}/\cS_i) \star \det (\cS_{i}/\cS_{i-1}) \\
 &{}= y^{r_{i+1}-r_{i-1}}Q_{i-1}Q_i\frac{\det\cS_{i+1}}{\det\cS_{i-1}} ,
 \end{align*}
 where the inverse is calculated in the quantum K ring of $X$. By~\eqref{eqn:relX} again,
 \[ \frac{\det\cS_{i+1}}{\det\cS_{i-1}} = \frac{\det \cS_{i+1} \star\det \C^n/\cS_{i-1}}{(1-Q_{i-1})\det \C^n}\qquad \text{in $\QK_T^{\loc(\mathbf{\hat r})}(X)$}, \]
 and its pushforward is
\begin{gather}\label{eqn:3}
 \frac{\det\cS_{i+1}}{\det\cS_{i-1}}\in \QK_T^{\loc(\mathbf{\hat r})}(Y).
\end{gather}
Note that by~\eqref{eqn:relY},
expression \eqref{eqn:3} is equal to
\begin{gather*}
 \frac{\det\left(\cS_{i+1}/\cS_{i-1}\right)}{1-Q_{i-1}}\qquad \text{in} \ \QK_T^{\loc(\mathbf{\hat r})}(Y).
\end{gather*}
Using \eqref{eqn:pushforwards-lambda} and \eqref{eqn:3-expand}, the projection formula, and Theorem~\ref{thm:kato-morphism}, it follows that
\begin{gather*}
 \pi_*\begin{vmatrix}
 B_{i-1} & 0\\
 1 & A_{i}
 \end{vmatrix}^\star=\pi_*\left(B_{i-1}\star B_i\right)=y^{r_{i+1}-r_{i-1}}\frac{Q_{i-1}}{1-Q_{i-1}}\det\left(\cS_{i+1}/\cS_{i-1}\right)=B'_{i-1}.
\end{gather*}
 For (d), we calculate
 \[ \begin{vmatrix}
 A_{i-1} & B_i\\1 & A_i
 \end{vmatrix}^\star
 = \begin{vmatrix}
 \lambda_y(\cS_{i}/\cS_{i-1}) + B_{i-1} & B_i\\1 & A_i
 \end{vmatrix}^\star
 = A_i \star \lambda_y(\cS_{i}/\cS_{i-1}) + A_i \star B_{i-1} - B_i. \]
 From (c), $\pi_* (A_i \star B_{i-1}) = B_{i-1}'$, therefore it suffices to show that
 $A_i \star \lambda_y(\cS_{i}/\cS_{i-1}) - B_i$ may be pushed forward, and that
 \begin{gather}\label{eqn:nts}
 \pi_*\left(A_i \star \lambda_y(\cS_{i}/\cS_{i-1}) - B_i\right) = \lambda_y(\cS_{i+1}/\cS_{i-1}) .
 \end{gather}
 The hypothesis $r_i - r_{i-1} =1$ implies that $\cS_i/\cS_{i-1}$ is a line bundle, and that
 \begin{gather*}
 A_i \star \lambda_y(\cS_{i}/\cS_{i-1}) - B_i\\
 \qquad{} = \lambda_y(\cS_{i+1}/\cS_{i}) \star \lambda_y(\cS_{i}/\cS_{i-1})
 + y^{r_{i+1}-r_{i-1}} \frac{Q_i}{1-Q_i} \det (\cS_{i+1}/\cS_{i}) \star \det (\cS_{i}/\cS_{i-1}) .
 \end{gather*}
 By \eqref{eqn:pushforwards-lambda}, we have
 \[
 \pi_*(\lambda_y(\cS_{i+1}/\cS_{i}))=\lambda_y(\cS_{i+1}/\cS_{i-1})-y^{r_i-r_{i-1}}\det(\cS_{i+1}/\cS_{i-1}), \qquad\pi_*(\lambda_y(\cS_{i}/\cS_{i-1}))=1.
 \]
 By \eqref{eqn:relX}, we have
 \[
 \frac{Q_i}{1-Q_i} \det (\cS_{i+1}/\cS_{i}) \star \det (\cS_{i}/\cS_{i-1}) = Q_i(1-Q_{i-1})\frac{\det\cS_{i+1}}{\det \cS_{i-1}}.
 \]
 As in the proof of (c), this can be pushed forward and its pushforward is $\det(\cS_{i+1}/\cS_{i-1})$. Putting these together, we have established \eqref{eqn:nts}.
\end{proof}

Recall that $Y^{(j)}=\bigl(Y^{(j)}_1,\dots,
{ Y^{(j)}_{r_{j+1}-r_j}}\bigr)$, $0\leq j\leq k$, are formal variables, $e_\ell$ denotes the $\ell$-th elementary symmetric polynomial, and $T_1,\dots, T_n \in \K_T(\pt)$ are given by
the decomposition of~$\C^n$ into one dimensional $T$-modules, that is, $\wedge^\ell(\C^n)=e_\ell(T_1,\dots,T_n)$.

\begin{Theorem}\label{thm:main}
 The ring $\QK_T(\Fl(\bfr,n))$ is isomorphic to $R\dbb{Q}/J_Q$, where
 \[
 R=\K_T(\pt)\bigl[e_1\bigl(Y^{(j)}\bigr),\dots,e_{r_{j+1}-r_j}\bigl(Y^{(j)}\bigr),\, 0\leq j\leq k\bigr],
 \]
 and $J_Q\subset R[\![Q]\!]=R[\![Q_1,\dots,Q_k]\!]$
is the ideal generated by the coefficients of $y$ in
\begin{align}\label{eqn:whit-short-gen}
 \prod_{\ell=1}^{n}(1+yT_\ell)-
 &\begin{vmatrix}
 A_0 & B_1 & & & \\
 1 & A_1 & B_2 & & \\ & \ddots & \ddots & \ddots & \\ & & 1 & A_{k-1} & B_{k}\\ & & & 1 & A_{k}
 \end{vmatrix},
 \end{align}
 where
 \[
 A_j=\prod_{\ell=1}^{r_{j+1}-r_j}\bigl(1+y Y^{(j)}_\ell\bigr)+B_j,\qquad B_j=y^{{r_{j+1}-r_j}}\frac{Q_j}{1-Q_j}\prod_{\ell=1}^{r_{j+1}-r_j}Y^{(j)}_\ell,
 \]
 with the convention that $Q_0=0$.

 More precisely,
there exists a $\K_T(\pt)[\![Q]\!]$-algebra isomorphism
\[
\Psi\colon\ R[\![Q]\!]/J_Q \to \QK_T(\Fl(r_1,\dots,r_k))
\]
that sends \smash{$e_\ell\bigl(Y^{(j)}\bigr)$} to $\wedge^\ell\br{\cS_{j+1}/\cS_j}$ for $j=0,\dots,k$ and $\ell=1,\dots,r_{j+1}-r_{j}$.
\end{Theorem}

\begin{proof}
It follows from Theorem~\ref{thm:whit-short} that $\Psi$ is a well-defined ring homomorphism.
To prove there are no other relations, we use~\cite[Theorem~4.1]{GMSXZZ:Nakayama}, which
states that a complete set of relations in the quantum (equivariant) K ring is
obtained by quantizing any complete set of relations in the ordinary (equivariant) K ring.
Therefore, we need to show that when one specializes each~$Q_i$ to~$0$, the resulting
ring is a presentation of $\K_T(\Fl(r_1,\dots,r_k))$.
The relations obtained this way
are the `Borel-type relations' of the $\lambda_y$ classes
\begin{equation}\label{E:pToda} \lambda_y(\cS_1) \cdot \lambda_y(\cS_2/\cS_1)\cdot \dots \cdot \lambda_y(\C^n/\cS_k) = \lambda_y(\C^n) . \end{equation}
Note that the relations~\eqref{E:pToda} can be obtained from the Whitney relations
\[ \lambda_y(\cS_i) \cdot \lambda(\cS_{i+1}/\cS_i) = \lambda_y(\cS_{i+1}) , \]
by eliminating the classes $\lambda_y(\cS_i)$ for $2 \le i \le k$.
(The quantization of this statement is done in the next section.)
Finally, it is known that
the Whitney relations form a full set of relations
in $\K_T(\Fl(r_1,\dots,r_k))$. This is essentially done by Lascoux~\cite[Section~7]{lascoux:anneau},
and we refer to~\cite[Proposition~5.1]{GMSXZZ:Nakayama} for a complete proof.
\end{proof}

We illustrate the proof of Theorem~\ref{thm:whit-short} with the following example.
\begin{Example}
 Let $\Fl(4)$ $\to \Gr(2,4)=\Fl(2;4)$ be the projection. In $\QK_T(\Fl(4))$, we have the following relation
 \begin{gather}\label{eqn:fl4}
 \lambda_y\bigl(\mathbb{C}^4\bigr)=\begin{vmatrix}
 A_0 & B_1 & 0 &0 \\
 1 & A_1 &B_2 &0 \\
 0 & 1 & A_2 & B_3 \\
 0 & 0 & 1 & A_3
 \end{vmatrix}^\star,
\end{gather}
where
\begin{gather*}
A_0 = \lambda_y(\mathcal{S}_1),\qquad B_1 = y\frac{Q_1}{1-Q_1}\det(\mathcal{S}_2/{\mathcal{S}_1}),
\\
A_1= \lambda_y(\mathcal{S}_2/{\mathcal{S}_1})+y\frac{Q_1}{1-Q_1}\det(\mathcal{S}_2/{\mathcal{S}_1}) ,\qquad B_2 = y\frac{Q_2}{1-Q_2}\det(\mathcal{S}_3/{\mathcal{S}_2}),
\\
A_2 = \lambda_y(\mathcal{S}_3/{\mathcal{S}_2}) + y\frac{Q_2}{1-Q_2}\det(\mathcal{S}_3/{\mathcal{S}_2}) ,\qquad B_3 = y\frac{Q_3}{1-Q_3}\det\bigl(\mathbb{C}^4/{\mathcal{S}_3}\bigr),
\\
A_3 = \lambda_y\bigl(\mathbb{C}^4/{\mathcal{S}_3}\bigr) + y\frac{Q_3}{1-Q_3}\det\bigl(\mathbb{C}^4/{\mathcal{S}_3}\bigr).
\end{gather*}
We push this relation forward to $\Gr(2,4)$ by pushing it forward to $\Fl(2,3;4)$ and then pushing forward from $\Fl(2,3;4)$ to $\Gr(2,4)$.
Let $\pi\colon \Fl(4) \to \Fl(2,3;4)$ be the projection. The relation on $\Fl(2,3;4)$ is given by
\begin{gather}\label{eqn:234}
 \lambda_y\bigl(\mathbb{C}^4\bigr)=\begin{vmatrix}
 A_0' & B_2 & 0 \\
 1 & A_2 &B_3 \\
 0 & 1 & A_3 \\
 \end{vmatrix}^\star,
\end{gather}
where
\[
A_0' = \lambda_y(\mathcal{S}_2).
\]
By expanding the determinant in~\eqref{eqn:fl4} along the columns containing $A_0$ and $A_1$, we obtain
\begin{gather}\label{eqn:tobepushed}
 \lambda_y\bigl(\mathbb{C}^4\bigr)=\begin{vmatrix}
 A_0 & B_1 \\
 1 & A_1 \\
 \end{vmatrix}^\star \begin{vmatrix}
 A_2 & B_3 \\
 1 & A_3 \\
 \end{vmatrix}^\star -A_0 \begin{vmatrix}
 B_2 & 0 \\
 1 & A_3 \\
 \end{vmatrix}^\star.
\end{gather}
By Lemma~\ref{lemma:main-push},
\[
\pi_*\left|\begin{matrix}
 A_0 & B_1 \\
 1 & A_1 \\
 \end{matrix}\right|^\star = A_0' , \qquad \pi_*A_0=1
 \]
and \smash{$\bigl|\begin{smallmatrix}
 A_2 & B_3 \\
 1 & A_3 \\
 \end{smallmatrix}\bigr|^\star $}, \smash{$\bigl|\begin{smallmatrix}
 B_2 & 0 \\
 1 & A_3 \\
 \end{smallmatrix}\bigr|^\star $}
 will not change under pushforward by $\pi$. Thus, by pushing forward~\eqref{eqn:tobepushed} we obtain
 \begin{gather*}
 \lambda_y\bigl(\mathbb{C}^4\bigr)=A_0' \begin{vmatrix}
 A_2 & B_3 \\
 1 & A_3 \\
 \end{vmatrix}^\star - \begin{vmatrix}
 B_2 & 0 \\
 1 & A_3 \\
 \end{vmatrix}^\star
\end{gather*}
which is the expansion of (2) along the first column. So the relation in $\QK_T(\Fl(4))$ pushes forward to the relation in $\QK_T(\Fl(2,3;4))$.

Now let $p\colon \Fl(2,3;4) \to \Gr(2,4)$ be the projection. In $\Gr(2,4)$ we have the following relation:
\begin{gather}\label{eqn:gr24}
 \lambda_y\bigl(\mathbb{C}^4\bigr)=\begin{vmatrix}
 A_0' & B_1'' \\
 1 & A_1'' \\
 \end{vmatrix}^\star ,
\end{gather}
where
\[
B_1'' = y^2\frac{Q_2}{1-Q_2}\det\bigl(\mathbb{C}^4/{\mathcal{S}_2}\bigr) ,\qquad A_1'' = \lambda_y(\mathbb{C}^4/\mathcal{S}_2)+ y^2\frac{Q_2}{1-Q_2}\det\bigl(\mathbb{C}^4/{\mathcal{S}_2}\bigr).
\]
By Lemma~\ref{lemma:main-push}, in~\eqref{eqn:234}, we have \smash{$p_*\bigl|\begin{smallmatrix}
 A_2 & B_3 \\
 1 & A_3 \\
 \end{smallmatrix}\bigr|^\star = A_1''$}, \smash{$p_*
 \bigl|\begin{smallmatrix}
 B_2 & 0 \\
 1 & A_3 \\
 \end{smallmatrix}\bigr|^\star=B_1'' $} and $A_0'$ will not change under the pushforward. Thus,~\eqref{eqn:234} pushes forward to~\eqref{eqn:gr24}.
\end{Example}

\section{Whitney implies Toda}\label{sec:whitney}
In this section, we consider a different presentation of the quantum K ring, named the
{\em quantum~K Whitney presentation}.
This presentation quantizes
 relations $\lambda_y(\cS_i) \cdot \lambda_y(\cS_{i+1}/\cS_i) = \lambda_y(\cS_{i+1})$
 satisfied by the tautological subbundles in $\K_T(\Fl(\bfr,n))$. Informally, the
 Whitney presentation contains more (geometric) information than the Toda presentation,
 as it involves more generators, corresponding to the $\lambda_y$
 classes of the tautological subbundles, and their quotients. In~contrast, the Toda presentation
 only involves the quotient bundles.

The quantization was conjectured in~\cite{gu2023quantum,GMSXZZ:QKW},
generalizing the conjectures from
\cite{Gu:2020zpg} for Grassmannians. These conjectures have been proved in~\cite{gu2022quantum}
for Grassmannians, and in~\cite{GMSXZZ:QKW} for~$\Fl(1,n-1;n)$ case. The general case was recently announced in~\cite{huq2024quantum} using the abelian/non-abelian correspondence.
We note that the results in~\cite{huq2024quantum} are logically independent on those from~\cite{maeno2023presentation}, which were used to obtain the Toda presentation in the previous section.

Our main result of this section is that eliminating the additional variables of the Whitney presentation yields the Toda presentation. As an aside, we note that the proof of Theorem~\ref{thm:main} can be easily modified to show that the quantum K Whitney presentation of $\Fl(\bfr,n)$ follows from that of $\Fl(n)$.
We leave the details of this proof to the reader.

In what follows, $T$ can be a maximal torus in $\GL_n$.
Let
\[
 {X}^{(j)}=\bigl(X^{(j)}_1,\dots,X^{(j)}_{r_j}\bigr) \qquad\text{and}\qquad {Y}^{(j)}=\bigl(Y^{(j)}_1,\dots,
Y^{(j)}_{r_{j+1}-r_j}\bigr)
\]
denote formal variables for $j=1,\dots,k$ and denote by $X^{(k+1)}\coloneqq (T_1,\dots,T_n)$ the equivariant parameters in $\K_T(\pt)$. Let \smash{$e_\ell\bigl({X}^{(j)}\bigr)$} and \smash{$e_\ell\bigl({Y}^{(j)}\bigr)$} be the $\ell$-th elementary symmetric polynomials in ${X}^{(j)}$ and ${Y}^{(j)}$, respectively. Define the ring
 \[
 S=\K_T(\pt)\bigl[e_1\bigl(X^{(j)}\bigr),\dots, e_{r_j}\bigl(X^{(j)}\bigr),e_1\bigl(Y^{(j)}\bigr),\dots,e_{r_{j+1}-r_j}\bigl(Y^{(j)}\bigr)\bigr]_{j=1}^k ,
 \]
 and the ideal $I_Q\subset S[\![Q]\!]=S[\![Q_1,\dots,Q_k]\!]$
generated by the coefficients of $y$ in
\begin{gather}\label{eqn:qkrel}
 \prod_{\ell=1}^{r_j}\bigl(1+y X^{(j)}_\ell\bigr)\prod_{\ell=1}^{r_{j+1}-r_j}\bigl(1+y Y^{(j)}_\ell\bigr)-\prod_{\ell=1}^{r_{j+1}}\bigl(1+y X^{(j+1)}_\ell\bigr)\\
 \quad{}+y^{{r_{j+1}-r_j}}\frac{Q_j}{1-Q_j}\prod_{\ell=1}^{r_{j+1}-r_j}Y^{(j)}_\ell\Biggl(\prod_{\ell=1}^{r_j}\bigl(1+yX^{(j)}_\ell\bigr) -\prod_{\ell=1}^{r_{j-1}}\bigl(1+yX^{(j-1)}_\ell\bigr)\Biggr), \qquad j=1,\dots, k. \nonumber
\end{gather}
It was conjectured in~\cite{gu2023quantum,GMSXZZ:QKW} and proved in
\cite{huq2024quantum} that
there is an isomorphism of $\K_T(\pt)[\![Q]\!]$-algebras
 \begin{equation}\label{E:QKWhitney}
 \Phi\colon\ {S[\![Q]\!]}/I_Q\to \QK_T\br{\Fl(\bfr,n)}
 \end{equation}
 sending
 \[ e_\ell\bigl(X^{(j)}\bigr) \mapsto \wedge^\ell(\cS_{j}) \qquad \textrm{and} \qquad e_\ell\bigl(Y^{(j)}\bigr) \mapsto \wedge^\ell(\cS_{{j+1}}/\cS_{j}) . \]
 We refer to this as the (quantum K) {\em Whitney presentation}.

 \begin{Proposition}\label{prop:whit-toda}
 There is a natural isomorphism
 \[ S[\![Q]\!]/I_Q \simeq R\dbb{Q}/J_Q ,\]
 obtained by eliminating the indeterminates \smash{$X^{(j)}_\ell$}. In particular,
 the Whitney relations from~\eqref{eqn:qkrel} imply the
 Toda relations from~\eqref{eqn:whit-short-gen}.
 \end{Proposition}
 \begin{proof}
 Let
 \begin{align*}
 A_j=\prod_{\ell=1}^{r_{j+1}-r_j}\bigl(1+y Y^{(j)}_\ell\bigr)+B_j,\qquad B_j=y^{{r_{j+1}-r_j}}\frac{Q_j}{1-Q_j}\prod_{\ell=1}^{r_{j+1}-r_j}Y^{(j)}_\ell,
 \end{align*}
 so that~\eqref{eqn:qkrel} becomes
 \begin{align}\label{eqn:recur}
 A_j\prod_{\ell=1}^{r_j}\bigl(1+yX^{(j)}_\ell\bigr) -B_j\prod_{\ell=1}^{r_{j-1}}\bigl(1+yX^{(j-1)}_\ell\bigr) -\prod_{\ell=1}^{r_{j+1}}\bigl(1+y X^{(j+1)}_\ell\bigr).
 \end{align}
 Note that by Lemma~\ref{L:r-det}, relations given by~\eqref{eqn:recur} are equivalent to those given by
 \begin{align*}
 \prod_{\ell=1}^{r_{j+1}}\bigl(1+y X^{(j+1)}_\ell\bigr)-
 &\begin{vmatrix}
 A_0 & B_1 & & & \\
 1 & A_1 & B_2 & & \\ & \ddots & \ddots & \ddots & \\ & & 1 & A_{j-1} & B_{j}\\ & & & 1 & A_{j}
 \end{vmatrix}\qquad \text{for $1\leq j\leq k$}.
 \end{align*}
 As a consequence, we can eliminate $e_1\bigl(X^{(j)}\bigr),\dots, e_{r_j}\bigl(X^{(j)}\bigr)$ for $2\leq j\leq k$, and be left with the relation~\eqref{eqn:whit-short-gen}.
 \end{proof}
\begin{Remark}\label{rmk:QKW} We note that our methods from the previous section can be adapted easily to show that $\Phi$ is an isomorphism for all partial flag varieties if and only if it is an isomorphism for $\Fl(n)$. \end{Remark}

 We illustrate Proposition~\ref{prop:whit-toda} with the following two examples.

 \begin{Example} Consider $\Fl(2)= \bP^1$ with the tautological subbundle
 $\cS_1 \subset \C^2$. The QK Whitney relations are given by
 \[
 \lambda_y(\cS_1) \star \lambda_y\bigl(\mathbb{C}^2/\cS_1\bigr)= \lambda_y\bigl(\mathbb{C}^2\bigr)-y\frac{Q}{1-Q}\bigl(\mathbb{C}^2/\cS_1\bigr) \star (\lambda_y(\cS_1)-1)
 \]
 After making the change of variables
 $\mathcal{S}_1 \mapsto P_1$
 and $\mathbb{C}^2/\cS_1 \mapsto (1-Q)P_2/{P_1}$, then
 collecting the coefficients of $y$ and $y^2$, one obtains
 the Toda relations for $\QK_T\bigl(\mathbb{P}^1\bigr)$:
 \[
 P_1 +\frac{1-Q}{P_1} = \mathbb{C}^2, \qquad P_2 = \wedge^2\mathbb{C}^2 .
 \]
 \end{Example}
 \begin{Example} We now consider the case $X=\Fl(3)$, equipped with the
 tautological sequence $\cS_1 \subset \cS_2 \subset \C^3$.
 There are two QK Whitney relations
 \begin{gather*}
 \lambda_y(\mathcal{S}_1) \star \lambda_y(\mathcal{S}_2/\mathcal{S}_1) = \lambda_y(\mathcal{S}_2) - y\frac{Q_1}{1-Q_1} \mathcal{S}_2/\mathcal{S}_1 \star (\lambda_y(\mathcal{S}_1)-1),
 \\
 \lambda_y(\mathcal{S}_2) \star \lambda_y\bigl(\mathbb{C}^3/\mathcal{S}_2\bigr) = \lambda_y\bigl(\mathbb{C}^3\bigr)-y\frac{Q_2}{1-Q_2}\mathbb{C}^3/{\mathcal{S}_2} \star (\lambda_y(\mathcal{S}_2)-\lambda_y(\mathcal{S}_1)) .
 \end{gather*}
 From the first relation, we can write
 \[\lambda_y(\mathcal{S}_2) = \lambda_y(\mathcal{S}_1) \star \lambda_y(\mathcal{S}_2/\mathcal{S}_1) + y\frac{Q_1}{1-Q_1} \mathcal{S}_2/\mathcal{S}_1 \star (\lambda_y(\mathcal{S}_1)-1),
 \]
 which we can use to replace $\lambda_y(\mathcal{S}_2)$ in the second relation. By some algebra, we obtain
 \begin{gather*}
 (1+y\mathcal{S}_1)\star(1+y\mathcal{S}_2/\mathcal{S}_1)\star\bigl(1+y\mathbb{C}^3/\mathcal{S}_2\bigr)+y^2\frac{Q_1}{1-Q_1}\mathcal{S}_2/\mathcal{S}_1 \star \mathcal{S}_1 \star \bigl(1+y\mathbb{C}^3/\mathcal{S}_2\bigr)\\
 \qquad{}=\lambda_y\bigl(\C^3\bigr)-y\frac{Q_2}{1-Q_2}\mathbb{C}^3/\mathcal{S}_2 \star (1+y\mathcal{S}_1)\star(1+y\mathcal{S}_2/\mathcal{S}_1)\\
 \qquad\quad{}- y^3\frac{Q_1Q_2}{(1-Q_1)(1-Q_2)}\mathcal{S}_1 \star \mathcal{S}_2/\mathcal{S}_1 \star \mathbb{C}^3/\mathcal{S}_2 + y\frac{Q_2}{1-Q_2}\mathbb{C}^3/\mathcal{S}_2\star(1+y\mathcal{S}_1).
 \end{gather*}
 With the change of variables
 \[ \mathcal{S}_1 \mapsto P_1, \qquad \mathcal{S}_2/\mathcal{S}_1 \mapsto (1-Q_1) P_2/P_1 , \qquad \mathbb{C}^3/\mathcal{S}_2 \mapsto ({1-Q_2})P_3/P_2 \]
 and equating the coefficients of $y$, $y^2$, $y^3$ in the two sides to obtain
 \begin{itemize}\itemsep=0pt
 \item coefficient of $y$: $P_1 + (1 - Q_1) P_2/P_1 + (1 - Q_2) P_3/P_2 =\C^3$;
 \item coefficient of $y^2$: $P_2 + (1-Q_1)P_3/P_1 + (1-Q_2) P_1P_3/P_2 =\wedge^2 \C^3$;
 \item coefficient of $y^3$: $P_3= \wedge^3 \C^3$.
 \end{itemize}
 These are the Toda relations for $\QK_T(\Fl(3))$, calculated from~\eqref{eqn:toda}.
 \end{Example}

\section[Representatives for quantum K Schubert classes in partial flag varieties]{Representatives for quantum K Schubert classes\\ in partial flag varieties}

The goal of this section is to use the pushforward technique to
obtain polynomial representatives of Schubert classes in the equivariant quantum K rings of
partial flag varieties. Our strategy is to push forward
the polynomials for the class of the point from $\QK_T(\Fl(n))$ to $\QK_T(\Fl(\bfr,n))$, then use the
(left) divided difference operators defined in~\cite{Mihalcea2022Left}
in the rings $\QK_T(\Fl(\bfr,n))$
to deduce a recursive procedure giving
the other polynomials. The left divided difference operators were
also used by Maeno, Naito, and Sagaki
\cite{maeno2023presentation} to prove that the
quantum double Grothendieck polynomials represent Schubert classes in the Toda
presentation of $\QK_T(\Fl(n))$.

We use a different generating set from loc.\ cit.,
the exterior powers of the tautological bundles, thus our representatives
live in the (quantum) Whitney presentation introduced in Section~\ref{sec:whitney}.
A~key feature of our polynomials, and unlike those from~\cite{maeno2023presentation},
is that they {\em do not} involve quantum parameters.

\subsection[Preliminaries on Schubert classes and quantum divided difference operators]{Preliminaries on Schubert classes\\ and quantum divided difference operators}
We start with recalling some basic facts about the Schubert classes and
quantum divided difference operators in the equivariant quantum K theory.

We need the formula for the class of the Schubert point, proved
in~\cite{maeno2023presentation}, which we later use
to find formulae for the other Schubert classes. To this aim, we briefly recall
the definition of the Schubert basis in the quantum K rings.

Regard $\Fl(n)$ as $\SL_n/B$, and let $W:=N_{\SL_n}T/T \simeq S_n$ be the Weyl group,
equipped with the length function $\ell\colon W \to \mathbb{N}$. It is a Coxeter group,
generated by simple reflections $s_i=(i,i+1)$ for $1 \le i \le n-1$.
Denote by $w_0 \in W$ be the longest element, so that $\dim \Fl(n) = \ell(w_0)$.
Let~${W_{\bfr} \le W}$ be the subgroup generated by the simple
reflections $s_i$ so that $i$ is not among the components of $\bfr$,
and let $W^{\bfr} \subset W$ be the set of minimal length representatives
for the cosets of~$W/W^{\bfr}$.

Set $B^- = w_0 B w_0 \subset \SL_n$,
the opposite Borel subgroup. For each $w \in W$, the flag
variety $\Fl(n)$ has a $T$-fixed point $e_w:= n_wB$, where
$n_w \in N_{\SL_n}T/T$ is any representative of $w$. The
(opposite) Schubert cell is
$X^{w,\circ}:= B^-.n_wB \subset \Fl(n)$, and it is
isomorphic to the affine space $\mathbb{A}^{\dim \Fl(n) - \ell(w)}$.
One can similarly define Schubert cells in any partial flag variety $\Fl(\bfr,n)$;
alternatively, the Schubert cells in $\Fl(\bfr,n)$ are the images of
the Schubert cells in $\Fl(n)$ under the ($\SL_n$-equivariant)
natural projection $\Fl(n) \to \Fl(\bfr,n)$.~The Schubert variety
$X^w$ is the (Zariski) closure of the corresponding Schubert cell.
Inclusion of Schubert varieties give the Bruhat (partial) order on the set $W^{\bfr}$,
\[ u W^{\bfr} \le v W^{\bfr} \Leftrightarrow X^u \supset X^v \textrm{ in } \Fl(\bfr,n) . \]

Now let
$\cO^w \in \K_T(\Fl(\bfr,n))$ be the K theory class given by the structure sheaf
of $X^w$. The Schubert cells give a stratification of $\Fl(n)$, and, more generally, of $\Fl(\bfr,n)$.
Then the classes $\cO^w$ form a basis for $\K_T(\Fl(\bfr,n))$ over $\K_T(\pt)$,
when $w$ varies in the quotient $W/W_{\bfr}$.
This implies (by definition) that the classes $\cO^w$ are a basis of
$\QK_T(\Fl(\bfr,n))$, over the ground ring $\K_T(\pt)[\![Q]\!]$.

As in~\cite{maeno2023presentation}, we identify $\K_T(\pt)$
with the group algebra $\Z[P]=\oplus_{\chi\in P}\Z {\rm e}^\chi$ of the
weight lattice \smash{$P=\sum_{i=1}^{n-1}\Z \varpi_i$} of $\SL_n$,
where $\varpi_i, 1\leq i\leq n-1$ are the fundamental weights. We also
set $\varpi_0=\varpi_n=0$, and $\epsilon_j=\varpi_j-\varpi_{j-1}$ for $1\leq j\leq n$.

In~\cite{Mihalcea2022Left}, left divided difference operators acting
on $\QK_T(\Fl(\bfr,n))$ (in fact on the equivariant quantum K ring of any
homogeneous space $G/P$) were constructed. These operators send Schubert classes to Schubert classes, and
were compatible with the quantum K product. We recall next the salient facts, see
Section~8.3 in loc.\ cit.\ for further details.

Regard $\Fl(\bfr,n)$ as $\SL_n/P_\bfr$, where $P_{\bfr}$ the parabolic group stabilizing the identity
partial flag. Left multiplication by a representative $n_w$ of an element $w\in W$ induces an automorphism of $\Fl(\bfr,n)$ which is equivariant with respect to the automorphism of $T$ given by $t\mapsto n_w t n_w^{-1}$. Pulling back along this automorphism of $\Fl(\bfr,n)$ gives a ring automorphism $w^L$ of $\K_T(\Fl(\bfr,n))$.
The following combines~\cite[Proposition~5.3, Lemma 5.4, and Proposition~5.5]{Mihalcea2022Left}.

\begin{Proposition}[Mihalcea--Naruse--Su]
The following hold:
\begin{enumerate}\itemsep=0pt
 \item[$1.$] $w^L({\rm e}^\chi a) = {\rm e}^{w(\chi)} w^L(a)$ for any ${\rm e}^\chi\in \K_T(\pt)$ and $a\in \K_T(\Fl(\bfr,n))$.
 \item[$2.$] $w^L$ is $\K_{\SL_n}(\Fl(\bfr,n))$-linear: for
$\kappa \in \K_{\SL_n}(\Fl(\bfr,n))$ and $a \in \K_T(\Fl(\bfr,n))$,
\[ w^L(\kappa \cdot a) = \kappa\cdot w^L(a) . \]
 \item[$3.$] $w^L$ commutes with the natural projection $\pi\colon \Fl(n) \to \Fl(\bfr,n)$:
\[ w^L(\pi_*(a)) = \pi_*\bigl(w^L(a)\bigr) , \qquad \forall a \in \K_T(\Fl(n)) . \]
In particular, the map $w^L$ on $\K_T(\Fl(\bfr,n))$ is determined by the map on $\K_T(\Fl(n))$.
 \item[$4.$] The automorphisms $w^L$ give an action of $W$ on $\K_T(\Fl(\bfr,n))$.
If $s_i \in W$ is a simple reflection, and $\cO^w \in \K_T(\Fl(\bfr,n)$, then
\begin{equation*}
 s_i^L (\cO^w)= \begin{cases} {\rm e}^{\alpha_i}\cO^w+(1-{\rm e}^{\alpha_i})\cO^{s_i w} & \textrm{if } s_i w W_{\bfr} < w W_{\bfr}, \\ \cO^w & \textrm{otherwise} , \end{cases} \end{equation*}
where $\alpha_i$ is the simple positive root giving $s_i$.
 \end{enumerate}
\end{Proposition}

The equivariant quantum K theory is functorial for isomorphisms. Thus
one may extend the action of $W$ to an action on $\QK_T(\Fl(\bfr,n))$
by $\mathbb{Q}[\![Q]\!]$-linear ring automorphisms.
Define the (quantum) left divided difference operators by
\begin{equation*}
 \delta_i := \frac{1}{1- {\rm e}^{-\alpha_i}}\bigl(\id - {\rm e}^{-\alpha_i} s_i^L\bigr) .\end{equation*}
\big(In~\cite[equation~(13)]{Mihalcea2022Left} this operator is denoted by $\delta_i^\vee$.\big)
These operators have the same properties as the ordinary Demazure operators,
and they satisfy a Leibniz rule compatible with the quantum~K product.
For reader's convenience, we state these properties next,
see~\cite[Proposition~8.3]{Mihalcea2022Left}.
\begin{Proposition}[Mihalcea--Naruse--Su]\label{prop:QKDem} \quad
 \begin{enumerate}\itemsep=0pt
 \item[$1.$] The quantum operators $\delta_i$ are $\Q[q]$-linear, satisfy the braid relations, and $(\delta_i)^2= \delta_i$.
 \item[$2.$] For each $w \in W^\bfr$,
 \[
 \delta_i \bigl(\cO^{wW_\bfr}\bigr)= \begin{cases} \cO^{s_i w W_\bfr} & \textrm{if } s_i w < w, \\ \cO^{wW_\bfr} & \textrm{otherwise} . \end{cases}
 \]
 \item[$3.$] $($Leibniz rule$)$ For any $a, b \in \QK_T(\Fl(\bfr,n))$,
 \[
 \delta_i(a \star b) = \delta_i(a) \star b + {\rm e}^{-\alpha_i} s_i^L(a) \star \delta_i(b) - {\rm e}^{-\alpha_i} s_i^L(a) \star s_i^L(b) .
 \]
 \item[$4.$] The operator $\delta_i$ is a $\QK_{\SL_n}(\Fl(\bfr,n))$-module homomorphism, that is, for any $\kappa$ from $\QK_{\SL_n}(\Fl(\bfr,n))$ and $\eta$ from $\QK_T(\Fl(\bfr,n))$,
\[ \delta_i(\kappa \star \eta) = \kappa \star \delta_i(\eta) . \]
 \end{enumerate}
\end{Proposition}

Part (1) implies that for each $w \in W$ there are well defined operators $\delta_w$ acting on
quantum K ring $\QK_T(\Fl(\bfr,n))$. Furthermore, part (2) implies that if $w \in W$ is a minimal length representative
in its coset in~$W/W_{\bfr}$, then
\[ \cO^w = \delta_{ww_0}\bigl(\cO^{w_0 W_{\bfr}}\bigr).\]

\subsection{Polynomial representatives} In this section, we use results of
\cite{maeno2023presentation} to obtain a formula for the class of the Schubert
point in~$\QK_T(\Fl(n))$. Then we use Kato's pushforward, and the left divided difference
operators $\delta_w$, to obtain a recursive formula for the Schubert classes in any $\QK_T(\Fl(\bfr,n))$.

To start, note that, in geometric terms, the relations~\eqref{eqn:qkrel} are interpreted as follows
(cf.~\cite{gu2023quantum,GMSXZZ:QKW,huq2024quantum}).

\begin{Theorem}
 For $j=1,\dots, k$, the following relations hold in $\QK_T(X)$:
 \begin{gather}
 \lambda_y(\cS_{j})\star\lambda_y(\cS_{{j+1}}/\cS_{j})\nonumber\\
 \qquad{}=\lambda_y(\cS_{{j+1}})-y^{r_{j+1}-r_j}\frac{Q_j}{1-Q_j}\det(\cS_{{j+1}}/\cS_{j})\star(\lambda_y(\cS_{j})-\lambda_y(\cS_{{j-1}})).\label{eqn:lambda_y_rel}
 \end{gather}
\end{Theorem}
\begin{Proposition}\label{prop:F}
 The following holds in $\QK_T(\Fl(n))$:
 \begin{gather}\label{eqn:F}
 \wedge^{p}\cS_k =\sum_{\substack{J\subseteq[k]\\|J|=p}}\Biggl(\prod_{\substack{1\leq j\leq k\\j,j+1\in J}}\frac{1}{1-Q_j}\Biggr)\biggl(\sideset{}{^\star}\prod_{j\in J}\cS_{j}/\cS_{j-1}\biggr)
 \end{gather}
 for $0\leq p\leq k\leq n$, where $\star$ means the quantum $\K$ product.
\end{Proposition}
\begin{proof}
 We use double induction on $p$, $k$, with $p=k=0$ case being clear. Assume that
 \begin{gather*}
 \wedge^{p'}\cS_{k'} =\sum_{\substack{J\subseteq[k']\\|J|=p'}}\Biggl(\prod_{\substack{1\leq j\leq k'\\j,j+1\in J}}\frac{1}{1-Q_j}\Biggr)\biggl(\sideset{}{^\star}\prod_{j\in J}\cS_{j}/\cS_{j-1}\biggr)
 \end{gather*}
 for all $(p',k')<(p,k)$, then considering the three cases for $J\subseteq [k]$: $k\notin J$, $k,k-1\in J$, $k\in J$ and $k-1\notin J$, we have
 \begin{gather*}
 \sum_{\substack{J\subseteq[k]\\|J|=p}}\Biggl(\prod_{\substack{1\leq j\leq k\\j,j+1\in J}}\frac{1}{1-Q_j}\Biggr)\biggl(\sideset{}{^\star}\prod_{j\in J}\cS_{j}/\cS_{j-1}\biggr)\\
 \qquad{}=\wedge^p\cS_{k-1}+\cS_k/\cS_{k-1} \star \biggl(\frac{1}{1-Q_{k-1}}\wedge^{p-1}\cS_{k-1}-\frac{Q_{k-1}}{1-Q_{k-1}}\wedge^{p-1}\cS_{k-2}\biggr)\\
 \qquad{}= \wedge^p\cS_k,
 \end{gather*}
 where the last equality follows from the Whitney relations~\eqref{eqn:F}.
\end{proof}

After harmonizing conventions, and using Proposition~\ref{prop:F}, the following is a restatement of~\cite[Proposition~3.1]{maeno2023presentation}.
 \begin{Corollary}\label{cor:pt-class}
 In $\QK_T(\Fl(n))$, we have
 \[
 \cO^{w_0}=\sideset{}{^\star}\prod_{i=1}^{n-1}\lambda_{-1}({\rm e}^{-\epsilon_{n-i}} \cS_i).
 \]
\end{Corollary}
We illustrate the corollary next.
\begin{Example} We take $n=2$, thus $\Fl(2) = \bP\bigl(\C^2\bigr)$. Fix $e_1$, $e_2$ to be a basis for $\C^2$. For simplicity we regard $\bP^1$ as $\GL_2/B$ with $T'=(\C^*)^2$ acting naturally, and then restrict this action to $\SL_2$.
With these conventions, the Schubert point is $X^{w_0} = \langle e_2 \rangle$, and
the localizations of $\cS= \cO_{\bP^1}(-1)$ at the fixed points $\bP(\langle e_i \rangle)$, $i=1,2$, are $\cS|_{\bP(\langle e_i \rangle)}={\rm e}^{\epsilon_i}$. Then one easily checks that
\[ \cO^{w_0} = 1 - {\rm e}^{-\epsilon_1} \cS . \]
\end{Example}
\begin{Theorem}\label{thm:w0}
 In $\QK_T(\Fl(\bfr,n))$, we have
 \begin{gather}\label{eqn:poly}
 \cO^{w_0}=\sideset{}{^\star}\prod_{i=1}^{k}\sideset{}{^\star}\prod_{j=r_i}^{r_{i+1}-1}\lambda_{-1}({\rm e}^{-\epsilon_{n-j}} \cS_i) .
 \end{gather}
\end{Theorem}
\begin{proof}
 Let $X=\Fl(r_1,\dots,r_k;n)$, $Y=\Fl\bigl(r_1,\dots, \widehat{r}_i,\dots, r_k;n\bigr)$, and let
 $\pi\colon X\to Y$ be the natural projection. Corollary~\ref{cor:pt-class} implies that the
 claim is true for $\Fl(n)$. By induction, we assume that~\eqref{eqn:poly} holds for
 $X$, and we compute its pushforward under $\pi$ using Kato's pushforward map
 from Theorem~\ref{thm:kato-morphism}.
 Note that all but the term including
 $\cS_i$ are pulled back from~$Y$. By~\eqref{eqn:lambda_y_rel}, we have
\begin{gather*}
 \lambda_{-1}({\rm e}^{-\epsilon_{n-j}} \cS_i)=\lambda_{-1}({\rm e}^{-\epsilon_{n-j}}\cS_{i-1})\star\lambda_{-1}({\rm e}^{-\epsilon_{n-j}} \cS_i/\cS_{i-1})\\
 \qquad{}+\frac{Q_{i-1}}{1-Q_{i-1}}{ (-{\rm e}^{-\epsilon_{n-j}})^{r_{i+1}-r_i}} \det(\cS_i/\cS_{i-1})\star(\lambda_{-1}({\rm e}^{-\epsilon_{n-j}}\cS_{i-1})-\lambda_{-1}({\rm e}^{-\epsilon_{n-j}}\cS_{i-2})),
\end{gather*}
where we used (the $\lambda$-ring formalism asserting) that
 $\lambda_{-1}({\rm e}^\chi \otimes E) = \lambda_{-{\rm e}^{\chi}}(E)$. Since the pushforward
\smash{$\pi_*\bigl(\wedge^j \cS_i/\cS_{i-1}\bigr)\!=\!0$} for any $j>0$ by Proposition~\ref{prop:relativeBWB},
$\pi_*\lambda_{-1}({\rm e}^{-\epsilon_{n-j}} \cS_i)=\lambda_{-1}({\rm e}^{-\epsilon_{n-j}}\cS_{i-1})$,
and the claim on $Y$ follows from the projection formula.
\end{proof}

We illustrate the formula in Theorem~\ref{thm:w0} in the case of $\Gr(2,4)$.
The Schubert classes in $\Gr(2,4)$ are typically indexed by partitions
in the $2 \times 2$ square; the dictionary to translate into the indexing by Weyl group elements is the following:
\begin{gather*}
\cO^{(1)}= \cO^{s_2 W_\bfr},\qquad \cO^{(2)}= \cO^{s_3 s_2 W_\bfr},\\
\cO^{(1,1)}= \cO^{s_1s_2 W_\bfr},\qquad \cO^{(2,1)}= \cO^{s_1s_3s_2 W_\bfr},\qquad \cO^{(2,2)}= \cO^{s_2 s_1 s_3 s_2 W_\bfr} .
\end{gather*}

\begin{Example}[Theorem~\ref{thm:w0} for $\Gr(2,4)$]
Denote by $\cS$ the tautological subbundle.
Using (for instance) a
localization argument, one calculates that
\[ \lambda_y(\cS) = (1+y {\rm e}^{\epsilon_1}) (1+y {\rm e}^{\epsilon_2}) \cO^{\varnothing} -y {\rm e}^{\epsilon_2}(1+y {\rm e}^{\epsilon_1})\cO^{(1)}- y {\rm e}^{\epsilon_1} \cO^{(1,1)} . \]
Thus for any weight $\chi$,
\[ \lambda_{-1}({\rm e}^{\chi} \cS) = 1 - {\rm e}^{\chi} \cS + {\rm e}^{2\chi} \wedge^2 \cS \/ \]
can be expanded into a combination of Schubert classes. Then one checks directly that
\[ \lambda_{-1}({\rm e}^{-\epsilon_2} \cS) \star \lambda_{-1}({\rm e}^{-\epsilon_1} \cS) = \cO^{(2,2)} . \]

The relevant multiplications are\footnote{These can be calculated for example with A. Buch's {\em Equivariant Schubert Calculator}, available at \url{https://sites.math.rutgers.edu/~asbuch/equivcalc/}.}
\begin{gather*}
\cO^{(1)} \star \cO^{(1)} = (1- {\rm e}^{\epsilon_3-\epsilon_2})\cO^{(1)} + {\rm e}^{\epsilon_3-\epsilon_2}\cO^{(2)} + {\rm e}^{\epsilon_3-\epsilon_2}\cO^{(1,1)}-{\rm e}^{\epsilon_3-\epsilon_2}\cO^{(2,1)}, \\
\cO^{(1)} \star \cO^{(1,1)} = (1- {\rm e}^{\epsilon_3-\epsilon_1})\cO^{(1,1)} + {\rm e}^{\epsilon_3-\epsilon_1}\cO^{(2,1)}, \\
\cO^{(1,1)} \star \cO^{(1,1)} = {\rm e}^{\epsilon_3+\epsilon_2-2 \epsilon_1} \cO^{(1,1)} -
 {\rm e}^{\epsilon_3+\epsilon_2-2 \epsilon_1} \cO^{(2,1)} - {\rm e}^{\epsilon_3-\epsilon_1}\cO^{(1,1)}
 + {\rm e}^{\epsilon_3-\epsilon_1}\cO^{(2,1)} \\
 \hphantom{\cO^{(1,1)} \star \cO^{(1,1)} =}{}
 - {\rm e}^{\epsilon_2-\epsilon_1}\cO^{(1,1)} +{\rm e}^{\epsilon_2-\epsilon_1}\cO^{(2,2)} + \cO^{(1,1)} .
 \end{gather*}
\end{Example}

Next we state the main result of this section.
Recall the Whitney presentation $\Phi\colon S\dbb{Q}/I_Q \to \QK_T(\Fl(\bfr,n)))$ from~\eqref{E:QKWhitney}.

\begin{Theorem}\label{thm:QKpolys} Let $\bfr = (r_1, \dots, r_k)$. Under the isomorphism $\Phi$, the elements
\[
\mathcal{G}_w(X):= \Phi^{-1}\Biggl(\delta_w\Biggl(\sideset{}{^\star}\prod_{i=1}^{k}\sideset{}{^\star}\prod_{j=r_i}^{r_{i+1}-1}\lambda_{-1}({\rm e}^{-\epsilon_{n-j}} \cS_i)\Biggr)\Biggr) \]
are sent to $W_{\bfr}$-symmetric polynomials in the variables
$X^{(j)}$ for $j=1, \dots , k$,
such that
\[ \Phi(\mathcal{G}_w(X)) = \cO^w \qquad \in \QK_T(\Fl(\bfr,n)) .\]
Furthermore, the polynomials $\mathcal{G}_w(X)$ are independent of the Novikov variables $Q_i$
for $1 \le i \le k$.\end{Theorem}
\begin{proof} This follows from Proposition~\ref{prop:QKDem}: polynomial representatives for all Schubert classes can be obtained by applying the quantum left divided difference operators $\delta_i$ to the identity~\eqref{eqn:poly} above. This process does not introduce any $Q$'s.
\end{proof}

The proposition may be interpreted as saying that the {\em same} polynomials representing
Schubert classes in $\K_T(\Fl(\bfr,n))$ also represent their quantizations in $\QK_T(\Fl(\bfr,n))$;
of course, the ideal of {\em relations} in $\QK_T(\Fl(\bfr,n))$ needs to be quantized.

We illustrate next the calculation of the polynomials representing Schubert classes in the ring
$\QK_T(\Gr(2,4))$.
 \begin{Example}
We use left divided difference operators to find polynomial representatives for all Schubert classes in $\QK_T(\Gr(2,4))$, knowing from Theorem~\ref{thm:w0} the representative for the Schubert point.

Recall that $\alpha_i = \epsilon_i - \epsilon_{i+1}$, and denote by $\cS$ the class
of the tautological subbundle. First, observe that $\delta_i \bigl({\rm e}^\chi \otimes \wedge^k \cS\bigr) = \delta_i ({\rm e}^\chi) \otimes \wedge^k \cS$ by Proposition~\ref{prop:QKDem}\,(4), and
\[
\delta_i ({\rm e}^{\chi} ) =
\begin{cases} {\rm e}^{\chi}, & s_i(\chi) = \chi, \\
{\rm e}^{\chi}\dfrac{1 - ({\rm e}^{-\alpha_i})^{1+ \langle \chi, \alpha_i^\vee \rangle}}{1-{\rm e}^{-\alpha_i}}, & otherwise. \end{cases}
\]
It follows that
\[
\delta_i \biggl({\rm e}^{-k \epsilon_j} \wedge^k \cS\biggr) =
\begin{cases}
{\rm e}^{-k \epsilon_j} \wedge^k \cS, & j\neq i,\ i+1 , \\
0, & j=i,\ k=1, \\
- {\rm e}^{-(\epsilon_i+ \epsilon_{i+1})} \wedge^2 \cS, & j=i,\ k=2 , \\
{\rm e}^{-k \epsilon_{i+1}}\bigl(1+ {\rm e}^{-\alpha_i} + \dots + {\rm e}^{-(k-1)\alpha_i}\bigr) \wedge^k \cS, & j=i+1,\ k \ge 1 .
\end{cases}
\]

By Theorem~\ref{thm:w0}, $\cO^{(2,2)}$ is equal to
\[
\lambda_{-1}({\rm e}^{-\epsilon_1} \cS) \star \lambda_{-1}({\rm e}^{-\epsilon_2} \cS) = \bigl(1 - {\rm e}^{-\epsilon_1} \cS + {\rm e}^{-2\epsilon_1} \wedge^2 \cS\bigr)\star \bigl(1 - {\rm e}^{-\epsilon_2} \cS + {\rm e}^{-2\epsilon_2} \wedge^2 \cS\bigr) .
\]
We have that \smash{$\delta_2\bigl(\cO^{(2,2)}\bigr)=\cO^{(2,1)}$}. We now calculate
\smash{$\delta_2\bigl(\cO^{(2,2)}\bigr)$} by means of the Leibniz rule from Proposition~\ref{prop:QKDem}. We obtain
\begin{gather*}
\cO^{(2,1)} = \lambda_{-1}({\rm e}^{-\epsilon_1} \cS) \star \bigl(1- {\rm e}^{-(\epsilon_2+\epsilon_3)} \wedge^2 \cS\bigr) ,\\
\cO^2 =\delta_1\bigl(\cO^{(2,1)}\bigr) = 1 -({\rm e}^{-\epsilon_1-\epsilon_2}+{\rm e}^{-\epsilon_2 - \epsilon_3}+{\rm e}^{-\epsilon_1 - \epsilon_3}) \wedge^2 \cS + ({\rm e}^{-\epsilon_1 - \epsilon_2 - \epsilon_3}) \cS \star \wedge^2 \cS ,\\
\cO^{(1,1)} = \delta_3\bigl(\cO^{(2,1)}\bigr) = \lambda_{-1}({\rm e}^{-\epsilon_1} \cS) , \\
\cO^{(1)} = \delta_1\bigl(\cO^{(1,1)}\bigr) =1- {\rm e}^{-(\epsilon_1+\epsilon_2)} \wedge^2 \cS , \\
\cO^\varnothing = \delta_1 \bigl(\cO^{(1)}\bigr) = 1 .
\end{gather*}
\end{Example}

Finally, we can rewrite the operators $\delta_i$ as operators $\rho_i$ acting on $\Z\bigl[T_1^{\pm1},\dots,T_n^{\pm1}\bigr]$ by
\begin{align*}
 \rho_i=\frac{T_i-T_{i+1}s_i}{T_i-T_{i+1}},
\end{align*}
where $s_i$ replaces each $T_j$ by $T_{s_i(j)}$, and further extend it to
\begin{align*}
 S\dbb{Q}=\Z\bigl[e_1\bigl(X^{(j)}\bigr),\dots, e_{r_j}\bigl(X^{(j)}\bigr),e_1\bigl(Y^{(j)}\bigr),\dots,e_{r_{j+1}-r_j}\bigl(Y^{(j)}\bigr)\bigr]_{j=1}^k\dbb{Q}\otimes \Z\bigl[T_1^{\pm1},\dots,T_n^{\pm1}\bigr]
\end{align*}
by \smash{$\Z\bigl[e_1\bigl(X^{(j)}\bigr),\dots, e_{r_j}\bigl(X^{(j)}\bigr),e_1\bigl(Y^{(j)}\bigr),\dots,e_{r_{j+1}-r_j}\bigl(Y^{(j)}\bigr)\bigr]_{j=1}^k\dbb{Q}$}-linearity. Given $w\in S_{n}$ with reduced expression $w=s_{i_1}\dots s_{i_l}$, we define
\begin{align*}
 \rho_w=\rho_{i_1}\dots\rho_{i_l}.
\end{align*}
Since the operators $\rho_i$ satisfy the braid relations, the operator $\rho_w$ doesn't depend on the choice of reduced expression. We may restate Theorem~\ref{thm:QKpolys} as follows.

\begin{Theorem}\label{thm:double-Grothendieck}
 For $w\in W^\bfr$, the isomorphism $\Phi\colon S\dbb{Q}/I_Q\to \QK_T(\Fl(\bfr,n))$ sends the class~of
 \[\rho_w\Biggl(\prod_{i=1}^{k}\prod_{j=r_i}^{r_{i+1}-1}\prod_{\ell=1}^{r_i}\bigl(1-T_{n-j}^{-1}X^{(i)}_\ell\bigr)\Biggr) \]
 to $\cO^w$.
\end{Theorem}
We have not seen similar polynomials in the study of quantum K theory of flag manifolds.

\appendix

\section[Toda relations from finite difference operators (after Anderson--Chen--Tseng)]{Toda relations from finite difference operators\\ (after Anderson--Chen--Tseng)}\label{sec:toda}

The proof of the Toda relations in~\cite{maeno.naito.sagaki:QKideal}
relies on Kato's earlier results~\cite{kato:loop}. For the quantum K~ring
$\QK_T(\Fl(n))$, there is another proof of these relations, using an argument
combining the results of Iritani, Milanov and
Tonita~\cite{Iritani:2013qka} with results of Givental
and Lee~\cite{givental_lee}. More precisely, it is shown in~\cite{Iritani:2013qka} that
the symbols of finite
difference operators annihilating the K-theoretic~$J$ function of a variety $X$ give relations in the quantum K ring of $X$. Givental and Lee's results from loc.\ cit.\ imply that the
K-theoretic~$J$ function of the complete flag variety is an eigenfunction of the
(finite difference) Toda Hamiltonians. This observation was made in the unpublished note~\cite{act:2017} of Anderson--Chen--Tseng, but removed from the published version of their paper. For the sake of completeness, we give a brief account below, and in the process fill in some of the details to make the argument complete.

We start with recalling the definition
of the K-theoretic $J$-function of the complete flag variety $X=\Fl(n)$.
Denote by $P_i= \wedge^i \cS_i$; it is known that these line bundles algebra generate~$\K_T(\Fl(n))$
over~$\K_T(\pt)$; see, e.g.,~\cite[Proposition~3.1]{GMSXZZ:Nakayama}. Furthermore,
the curve classes associated to the Novikov variables
$Q_i$ are dual to the classes
$c_1\bigl(\wedge^i \cS_i^*\bigr)$.
For a fixed effective (multi)degree~${d \in H_2(X)}$, let $L$ be the cotangent line bundle at the unique marked point on the moduli space $\overline{M}_{0,1}(X,d)$.
Let also $\phi^\alpha$, $\phi_\alpha$ denote Poincar\'e-dual bases for $\K_T(X)$. (For example,
one may take Schubert classes $\cO^w$, and their duals -- the ideal sheaves of the boundary of the opposite Schubert varieties.)
The small $J$-function of $X$, denoted by $J_X$, is defined by
\[
J_X(q)\coloneq (1-q)\prod_i P_i^{\frac{\ln(Q_i)}{\ln(q)}}\sum_{d,\alpha} Q^d\biggl\langle \frac{\phi_\alpha}{1-qL}\biggr\rangle_{0,1,d}\phi^\alpha .
\]
We will explain later the meaning and the effect of the factor
\smash{$P_i^{\ln(Q_i)/\ln(q)}$}. We also note that the presence
of the prefactors $(1-q)$
and \smash{$\prod_i P_i^{\ln(Q_i)/\ln(q)}$} varies in the literature.
Our description agrees with the one used by Givental and Lee in
\cite{givental_lee}, and corresponds to the function denoted by $\tilde{J}$ in~\cite{Iritani:2013qka}.

We recall some basics on the formalism of difference operators. Consider commuting
variables $q, x_1, \dots, x_n$, and define the difference operators
\[
T_i:= q^{x_i \partial_{x_i}} = \sum_{ k \ge 0}^\infty \frac{1}{k!}((\ln q)x_i \partial_{x_i} )^k.
\]
\big(More generally, for a differential operator $\mathfrak{f}$, one defines the $q$-difference operator $q^{\mathfrak{f}}={\rm e}^{(\ln q) \mathfrak{f}}=\sum_{j=0}^\infty \frac{1}{j!}((\ln q)\mathfrak{f})^j$.\big) Note that
\[
T_i \bigl(x_j^{\pm 1}\bigr) = \sum_{k=0}^\infty \frac{1}{k!} ( \ln(q) x_i \partial_{x_i} ) \bigl(x_j^{\pm 1}\bigr) = q^{\pm \delta_{ij}} x_j^{\pm 1} ,
\]
which explains the `difference operator' terminology. More generally, for any Laurent polynomial in commuting variables $x_i$, we have
\[
T_i f(x_1,\dots,x_i,\dots,x_n) = f(x_1,\dots,qx_i,\dots, x_n) ,
\]
i.e., $T_i$ is an automorphism of the Laurent polynomial ring
$\Z\bigl[q^{\pm 1}; x_1^{\pm 1}, \dots, x_{n}^{\pm 1}\bigr]$.
We use this expression to extend the definition of $T_i$ to any function in the indeterminates $q, x_1, \dots, x_n$.

Now consider the subring of Laurent polynomials
\[
\Z\bigl[q^{\pm 1}; Q_1^{\pm 1}, \dots, Q_{n-1}^{\pm 1}\bigr] \hookrightarrow \Z\bigl[q^{\pm 1}; x_1^{\pm 1}, \dots, x_{n}^{\pm 1}\bigr]
\]
obtained by sending \smash{$Q_i \mapsto q^{-1} \frac{x_{i+1}}{x_i}$}. The restriction of $T_i$ to this subring is given by
\begin{equation}\label{E: TiQ}
T_i = q^{-Q_{i}\partial_{Q_{i}}}q^{Q_{i-1}\partial_{Q_{i-1}}} ,
\end{equation}
where \smash{$q^{Q_i \partial_{Q_i}}$} are the difference operators on the subring in $Q_i$'s.

With that in mind, we can now explain the meaning of the factor \smash{$P^{\frac{\ln(Q_i)}{\ln(q)}}$}.
The difference operators \smash{$q^{Q_i \partial_{Q_i}}$} act on functions in $Q_i$'s, and
one calculates that
\[ q^{Q_i \partial_{Q_i}}\bigl(P^{\frac{\ln(Q_j)}{\ln(q)}}\bigr) =
P^{\frac{\ln(q^{\delta_{ij}} Q_j)}{\ln(q)}} = P^{\delta{ij}}
P^{\frac{\ln(Q_j)}{\ln(q)}} . \]
In other words, the factor \smash{$P^{\frac{\ln(Q_i)}{\ln(q)}}$} should be regarded as
a formal variable
which transforms according to the rule above under the difference operators.

The relations in the quantum K ring are given by the Hamiltonians of the finite difference
(or relativistic) Toda lattice. There is some ambiguity in the exact expressions for the
Toda Hamiltonians, since their construction depends on choices;
see, e.g.,~\cite[Remark~5]{givental_lee}. We follow here the approach from~\cite{gloI},
but we will also need to make some changes of variables,
in order to fit with the conventions in our main reference~\cite{givental_lee}.
For the convenience of the reader, we briefly included some of the details below.

The Hamiltonians of the $q$-deformed type $A$ Toda chain have the form
\begin{align}\label{eqn:hamil}
 H_k=\sum_{0= i_0<\dots<i_k\leq n} \prod_{l=1}^{k}\biggl(1-\frac{x_{i_l}}{x_{i_l-1}}\biggr)^{1-\delta_{i_l-i_{l-1},1}} \prod_{l=1}^{k} T_{i_l},\qquad k=1,\dots,n,
\end{align}
where $q$ and $x_i$ are commuting variables, and $T_i=q^{x_i\partial_{x_i}}$ is the $q$-difference operator above. It was proved in~\cite{gloI} that the operators $H_k$ are limits of Macdonald operators, and the latter are known to commute.
This implies that $H_k$ also commute.

As above, let $Q_i=q^{-1}{x_{i+1}}{x_i}^{-1}$, with $Q_0=Q_n=0$. Then, using
\eqref{E: TiQ}, one can rewrite~\eqref{eqn:hamil} as
\begin{align*}
 H_k=\sum_{0= i_0<\dots<i_k\leq n} \prod_{l=1}^{k}(1-qQ_{i_l-1})^{1-\delta_{i_l-i_{l-1},1}} \prod_{l=1}^{k} q^{-Q_{i_l}\partial_{Q_{i_l}}}q^{Q_{i_l-1}\partial_{Q_{i_l-1}}},\qquad k=1,\dots,n.
\end{align*}
Replacing $q$ by $q^{-1}$, we obtain
\begin{align*}
 \widehat{H}_k&{}= \sum_{0= i_0<\dots<i_k\leq n}\prod_{l=1}^{k}\bigl(1-q^{-1}Q_{{i_l}-1}\bigr)^{1-\delta_{i_l-i_{l-1},1}} \prod_{l=1}^{k} q^{Q_{i_l}\partial_{Q_{i_l}}-Q_{i_l-1}\partial_{Q_{{i_l}-1}}}\\
 &{}=\sum_{0= i_0<\dots<i_k\leq n} \prod_{l=1}^{k} q^{Q_{i_l}\partial_{Q_{i_l}}-Q_{i_l-1}\partial_{Q_{{i_l}-1}}}\prod_{l=1}^{k}(1-Q_{{i_l}-1})^{1-\delta_{i_l-i_{l-1},1}}
 ,\qquad k=1,\dots,n .
 \end{align*}
\begin{Remark} The substitutions above
ensure that the first Hamiltonian $\widehat{H}_1$ agrees
with the one used in~\cite{givental_lee}. The substitution chosen in~\cite{act:2017}
produces similar
operators, but with the $q$-shifts and the Novikov terms in the opposite order.
\end{Remark}
The following key result of Givental and Lee~\cite[Theorem~2]{givental_lee} shows that the $J$ function is an eigenfunction for $J_{\Fl(n)}$.

\begin{Theorem}[Givental--Lee]\label{thm:GLeemain}
 $\widehat{H}_1J_{\Fl(n)}=\mathbb{C}^n J_{\Fl(n)}$.
\end{Theorem}

We also need the following lemma of Givental--Lee~\cite{givental_lee}.

\begin{Lemma}[{\cite[p.~9]{givental_lee}}]\label{lemma:modQ}
Let $D$ be a difference operator commuting with \smash{$\widehat{H}_1$}. Then, if $J$ is an eigenfunction of $D$ modulo $Q$, then $J$ is an eigenfunction of $D$ whose eigenvalue is the same as the one modulo $Q$.
\end{Lemma}

From this, we deduce that $J_{\Fl(n)}$ is an eigenfunction of the higher Toda Hamiltonians, using their commutativity with \smash{$\widehat{H}_1$} and by computing their eigenvalues modulo $Q$.

\begin{Corollary}\label{cor:Hk-eigenval} For any $1 \le k \le n$, the following holds:
\[ \widehat{H}_kJ_{\Fl(n)}=\wedge^k(\mathbb{C}^n) J_{\Fl(n)} . \]
\end{Corollary}
\begin{proof} The case $k=1$ is Theorem~\ref{thm:GLeemain}. Suppose $2 \le k \le n$. Since \smash{$\widehat{H}_k$} commutes with \smash{$\widehat{H}_1$}, we need only verify that $J_{\Fl(n)}$ is an eigenfunction of \smash{$\widehat{H}_k$} modulo $Q$, thanks to Lemma~\ref{lemma:modQ}.

To this end, we first observe
\begin{align*}
\widehat{H}_k J_{\Fl(n)}&{}=\widehat{H}_k\biggl((1-q)\prod_i P_i^{\frac{ln(Q_i)}{ln(q)}}\biggr) + o(Q_i) \\
&{}=
\sum_{0= i_0<\dots<i_k\leq n}\prod_{l=1}^k \frac{P_{i_l}}{P_{i_l-1}}\biggl((1-q)\prod_i P_i^{\frac{ln(Q_i)}{ln(q)}}\biggr)+o(Q_i) .
\end{align*}
Thus, modulo $Q$, we have the eigenvalue equation
\begin{align*}
\widehat{H}_k J_{\Fl(n)}&{}=\sum_{0= i_0<\dots<i_k\leq n}
\prod_{l=1}^k \frac{P_{i_l}}{P_{i_l-1}}J_{\Fl(n)}\\
&{}=
e_k\biggl(\frac{P_1}{P_0},\frac{P_2}{P_1},\dots,\frac{P_{n}}{P_{n-1}}\biggr)J_{\Fl(n)}=\wedge^k(\mathbb{C}^n)J_{\Fl(n)} .
\tag*{\qed}
\end{align*}
\renewcommand{\qed}{}
\end{proof}

We now use~\cite[Proposition~2.12]{Iritani:2013qka} which shows that the symbols of
Toda Hamiltonians give relations in quantum K theory.

\begin{Theorem}[Iritani--Milanov--Tonita]\label{thm:IMT}
Let $D=D\bigl(q^{Q_i\partial_{Q_i}},q,Q,\Lambda_i\bigr)$ be any $q$-difference operator
with coefficients in $\K_T(\pt)[q^{\pm 1}][\![Q_i]\!]$, such that it is regular
at $q=1$. Then
\[DJ_X=0\implies D\bigl(\widehat{P}_i,1,Q,\Lambda_i\bigr)=0 \in \QK_T(X) .\]
\end{Theorem}

\begin{Remark}
The result of Iritani, Milanov and Tonita is stated non-equivariantly, and for the
{\em big} quantum K ring and the corresponding big $J$ function. However, an inspection of their proof shows that it works in the equivariant situation as well.
Furthermore, if one starts with the small quantum K ring, then all arguments extend
to that situation, and the result also holds for the small quantum K ring and the small $J$ function. For further details, see~\cite{huq2024relations}.
\end{Remark}
One subtle point is that $\widehat{P}_i$ is a certain $Q$-deformation
of the line bundle $P_i$: it is the restriction to the small
quantum K ring
of an operator denoted by
$A_{i,{\rm com}}$ in~\cite[Corollary~2.9]{Iritani:2013qka}, which
arises as a solution to a certain Lax-type equation.
However, results of both Anderson, Chen, Tseng, and Iritani
in~\cite[Lemma~6]{anderson2022finite}, and also
by Kato in~\cite[Theorem~1.35]{kato:loop}
show that in fact no quantization is needed.

\begin{Proposition}\label{thm:ACTI-Kato} For the flag variety $\Fl(n)$,
$\widehat{\det(\mathcal{S}_i)}=\det(\mathcal{S}_i)$.
\end{Proposition}
We note in passing that an analogue of Proposition~\ref{thm:ACTI-Kato} holds for any homogeneous
space~$G/P$, but we do not need this generality here.

Combining Theorem~\ref{thm:IMT} and Proposition~\ref{thm:ACTI-Kato} with Corollary~\ref{cor:Hk-eigenval}
yields the following corollary.
\begin{Corollary}
The following identities hold
in $\QK_T(\Fl(n))$:
\begin{equation*}
 \sum_{0= i_0<\dots<i_k\leq n}\ \prod_{l=1}^{k}\frac{P_{i_l}}{P_{i_{l}-1}}(1-Q_{i_l-1})^{1-\delta_{i_l-i_{l-1},1}}=\wedge^k \C^n,\qquad k=1,\dots,n,
\end{equation*}
where $P_0=P_n=1$.
\end{Corollary}

\begin{Remark}
Theorem 4.9 of~\cite{koroteev} gives a presentation of the quasimap quantum $K$-ring of $T^*Fl$ whose limit to the is described in Theorem 5.5. The relations are based on the trigonometric Ruijsenaars--Schneider model. After further taking into account a restriction from $\GL_n$ to $\SL_n$, the `Toda limit' recovers the relations in this paper. We are grateful to P. Koroteev who explained this procedure to us.
\end{Remark}

\subsection*{Acknowledgements}

The authors thank Dave Anderson, Linda Chen, Takeshi Ikeda, Shinsuke Iwao, Peter Koroteev, Takafumi Kouno, Satoshi Naito, Daisuke Sagaki, Mark Shimozono, and Kohei Yamaguchi for useful discussions, and sharing insights related to this work. L.M.\ was partially supported by NSF grant DMS-2152294, and gratefully acknowledges the support of Charles Simonyi Endowment, which provided funding for the membership at the Institute of Advanced Study during the 2024-25 Special Year in `Algebraic and Geometric Combinatorics'. D.O.\ gratefully acknowledges support from the Simons Foundation. Finally, we are grateful
to two anonymous referees for their valuable suggestions, which helped us improve the exposition of this paper.

\pdfbookmark[1]{References}{ref}
\LastPageEnding

\end{document}